\documentclass[11 pt,oneside,reqno,a4paper]{amsart} 
\usepackage{amsfonts,amssymb,amscd,amsmath, enumerate, verbatim, calc} 
\usepackage{float}
\usepackage{amsthm}
\newtheorem{theorem}{Theorem}[section] 
\newtheorem{definition}{Definition}[section]
\usepackage[all]{xy}
\usepackage{tikz-cd}
\newtheorem{lemma}[theorem]{Lemma}
\newtheorem{corollary}[theorem]{Corollary}
\newtheorem{proposition}[theorem]{Proposition}
\newtheorem{remark}[theorem]{Remark}
\usepackage{mathtools}
\newtheorem{problem}[theorem]{Problem}
\newtheorem{example}[theorem]{Example}
\numberwithin{equation}{section}
\usepackage[pagewise]{lineno}
\usepackage{amsfonts}
\newcommand{\field}[1]{\mathbb{#1}}          
\usepackage{cite}
\usepackage {hyperref}

\newcommand{\R}{\field{R}}                   
                   
\newcommand{\F}{\field{F}}

\renewcommand{\ker}{\textnormal{ker}}

\begin{document}
	
	\title{Theory of Extension of Commutative Poisson Algebra}
		\author{ Neeraj Kumar Maurya$^{1}$, Deepak Pal$^{2},$ and Sumit Kumar Upadhyay$^{3}$ \vspace{.4cm}\\
				{Department of Applied Sciences,\\ Indian Institute of Information Technology Allahabad\\Prayagraj, U. P., India} }
	
		\thanks{$^1$neerajkr2699@gmail.com, $^2$deepakpal5797@gmail.com, $^3$upadhyaysumit365@gmail.com}
    
	\thanks{Corresponding author: Deepak Pal.}
	\maketitle
    \textbf{Abstract}. In this paper, we develop the extension theory of commutative Poisson algebras. Specifically, we define the second cohomology of a commutative Poisson algebra $K$ with coefficients in an abelian Poisson algebra $H$. We then establish a natural bijective correspondence between this cohomology group and the set of equivalence classes of a special type of extensions.
    
\section{Introduction} 
A Poisson algebra $P$ over a field $\mathbb{F}$ is an associative algebra over $\mathbb{F}$ equipped with a Lie bracket $[-, -]$ satisfying the Leibnitz rule.
The origins of Poisson algebras lie in classical mechanics, where the Poisson bracket, introduced by Siméon Denis Poisson $(1809)$, was developed as a bilinear operation on smooth functions in phase space. In the study of Hamiltonian mechanics and mathematical physics, the Poisson algebra plays a crucial role. Numerous studies have been conducted by researchers in various areas of mathematics, including abstract algebra, representation theory, algebraic geometry, and differential geometry, on the subject of Poisson algebra. Today, Poisson algebras serve as a unifying framework in which Lie algebra theory interacts with commutative algebra, providing essential tools for studying deformation theory, representation theory, and the transition from classical to quantum mechanics.

Apart from their significant applications in the aforementioned fields, Poisson algebras have attracted considerable attention from an algebraic perspective owing to their rich structural properties and intrinsic mathematical interest. Consequently, the theory of Poisson algebras has been extensively developed in the literature; see \cite{AM0, BY, BY1, CM, CM1, PC, JM, FL, GR, DA, FK, FK1, LW, ML, UU} and the references therein. Extension theory has long been recognized as a fundamental area of research in algebra, playing a crucial role in the construction, analysis, and classification of algebraic structures. Since the pioneering work of Eilenberg and MacLane on group extensions and cohomology \cite{EM}, extension theory has subsequently been developed for a wide variety of algebraic systems, including groups \cite{B1}, Lie algebras \cite{HS, NR}, associative algebras \cite{MG}, Leibniz algebras \cite{LP}. These developments have employed powerful tools such as cohomology, crossed products, semidirect products, bicrossed products, and unified products to classify and construct algebraic extensions. In the context of Poisson algebras, Agore and Militaru \cite{AM,AM1} developed a comprehensive framework for extension theory by introducing crossed and unified products together with suitable non-abelian cohomological invariants for the classification of extensions. Their work provides a systematic approach to the construction and classification of Poisson and Jacobi algebra extensions.

There is a concept of multiplicative Lie algebra, which was introduced by Ellis \cite{GJ}  in 1993. It generalizes the concepts of group and Lie algebra. Since every Poisson algebra is a Lie algebra, and every Lie algebra is a multiplicative Lie algebra, it follows that every Poisson algebra naturally gives rise to a multiplicative Lie algebra structure. In this article, we attempt to study the extension theory of Poisson algebras to tackle the problem of classification of all Poisson algebras $P$ (up to isomorphism) with a given ideal $H$ such that $P/H \cong K$, similar to the study of the extension theory of groups \cite{RL1},  Lie algebras \cite{RL} and multiplicative Lie algebras \cite{MS}. 

In this paper, we extend these ideas to commutative Poisson algebras, where the presence of three binary operations makes the theory more intricate than in the case of Lie algebras. We establish an equivalence between extensions and factor systems of commutative Poisson algebras, and investigate the existence of ring center extensions associated with a homomorphism  $\Gamma: K \to OutDer(H).$ Furthermore, we introduce a second cohomology group tailored to this setting and prove a bijective correspondence between it and the set of equivalence classes of ring center extensions. 

\section{Preliminaries}
    This section reviews fundamental notions of commutative Poisson algebras. We refer \cite{VR} for the fundamentals of the Poisson algebra. 
\begin{definition}
A vector space $V$ over a field $\mathbb{F}$ with a binary operation $\cdot$ on $V$ is called an associative $\F$-algebra over $\F$ if the following conditions hold for all  $x,y,z \in V,$ and $\alpha \in \F :$
\begin{enumerate}
\item $(x \cdot y) \cdot z = x \cdot (y \cdot z),$ 
\item $(x+y) \cdot z = x \cdot z + y \cdot z,$ 
\item $\alpha(x \cdot y) = (\alpha x) \cdot y = x \cdot (\alpha y).$
\end{enumerate}
\end{definition}
\begin{definition}
Let $H$ and $K$ be vector spaces over a field $\mathbb{F}$.  
\begin{enumerate}
    \item A bilinear map $f: K \times K \to H$ is a function satisfying  
    \[
    f(\alpha x + \beta y, z) = \alpha f(x,z) + \beta f(y,z)
    \]
    and
    \[
    f(x, \alpha y + \beta z) = \alpha f(x,y) + \beta f(x,z)
    \]
    for all $x,y,z \in K$ and $\alpha, \beta \in \mathbb{F}$.

    \item A bilinear map $f$ is symmetric if $f(x,y) = f(y,x)$ for all $x,y \in K$.

    \item A bilinear map $h: K \times K \to H$ is alternating if $h(x,x) = 0$ for all $x \in K$.
\end{enumerate}
\end{definition}

\begin{definition} A Lie algebra  is a vector space $L$ over a field $\mathbb{F}$, together with a binary operation $[\ ,\ ]: L \times L \rightarrow L$ called the Lie bracket, satisfying the following identities for all $x,y,z \in L,$ and $\alpha,\beta \in\mathbb{F:}$ 
\begin{enumerate}
    \item  $[\alpha x + \beta y, z]= \alpha[x,z] + \beta[y,z]$, and	\\ $[x,\alpha y+\beta z]= \alpha [x,y] + \beta [x,z]$,
\item   $[x,x]=0$,
\item  $[x,[y,z]] + [y,[z,x]] + [z,[x,y]]=0$.
\end{enumerate}
\end{definition}
	
\begin{definition} \label{D1.3} 
A commutative Poisson algebra $P$ is a vector space  over a field $\mathbb{F}$, with two binary operations $\cdot$ and $[\ ,\ ]$ such that the following conditions hold:
\begin{enumerate}
			\item $(P,\cdot)$ is a commutative associative $\mathbb{F}$-algebra.
			\item $(P,[\ ,\ ])$ is a Lie algebra over $\mathbb{F}$ (where $[\ ,\ ]$ is called a Lie bracket).
			\item $[x\cdot y,z]=[x,z]\cdot y + x\cdot [y,z]$ for all $x,y,z$ in $P$.
		\end{enumerate}
		The Lie bracket is then called a Poisson bracket.
	\end{definition}
	\textbf{Note:} The Poisson algebra $P$ also satisfies the condition $$ [x, y \cdot z]=[x,y]\cdot z + y\cdot [x,z]$$ for all $x,y,z$ in $P.$  
         
	\begin{definition} 
	 Let $P$ be a commutative Poisson algebra.
		\begin{enumerate}
			\item  A subspace $H$ of $P$ is said to be a subalgebra of $P$ if $HH\subseteq H$ and $[H, H]\subseteq H.$  
			\item A subalgebra $H$ of $P$ is said to be an ideal of $P$ if $HP\subseteq H$ and $[H, P]\subseteq H.$
            
			\item The Poisson center of $P$ is defined by $$Z(P) = \{ \ x\in P \ | \ xy = [x, y] = 0, \ \text{for all }y \in P  \ \}.$$It is an ideal of $P.$
			\item A commutative Poisson algebra $P$ is said to be an abelian Poisson algebra if $P$ is a commutative Poisson algebra with a trivial product and a trivial bracket operation, that is, $PP = 0 = [P, P] $. Hence, a commutative Poisson algebra $P$ is said to be abelian if and only if $P=Z(P)$.
			   
		\end{enumerate}
	\end{definition}
\begin{definition} 
\begin{enumerate}
\item A short exact sequence $$\mathcal{E}(H, K) \equiv	{0}\longrightarrow  H\overset{\alpha} \longrightarrow P\overset{\beta} \longrightarrow  K \longrightarrow {0}$$ of commutative Poisson algebras is called an extension of $H$ by $K$. A vector space homomorphism $t: K \to P$ is called a section of  $\mathcal{E}(H, K)$ if $\beta \circ t= I_K$.
    \item A morphism from an extension $$\mathcal{E}(H, K) \equiv	{0}\longrightarrow  H\overset{\alpha} \longrightarrow P\overset{\beta} \longrightarrow  K \longrightarrow {0}$$ to an extension $$\mathcal{E'}(H', K') \equiv	{0}\longrightarrow  H'\overset{\alpha'} \longrightarrow P'\overset{\beta'} \longrightarrow  K' \longrightarrow {0}$$ is a triple $(\lambda, \mu, \nu)$, where $\lambda: H \to H'$, $\mu: P \to P'$ and $\nu: K \to K'$ are Poisson algebra homomorphisms such that the following diagram 
$$ 	
	\xymatrix{\mathcal{E}(H, K) \equiv 0 \ar[r] & H \ar[d]_-{\lambda} \ar[r]^-{i} & P \ar[d]_-{\mu} \ar[r]^-{\beta} & K \ar[d]^-{\nu} \ar[r] & 0 \\
		\mathcal{E'}(H', K') \equiv  0 \ar[r] & H' \ar[r]_-{i'} & P' \ar[r]_-{\beta'} & K' \ar[r] & 0
	}$$
    is commutative.
    \item Two extensions $\mathcal{E}(H, K)$ and $\mathcal{E'}(H, K)$ of $H$ by $K$ are equivalent if there is an isomorphism $\phi: P \to P'$ such that the following diagram
$$ 	\xymatrix{\mathcal{E}(H, K) \equiv 0 \ar[r] & H \ar[d]_-{I_H} \ar[r]^-{i} & P \ar[d]_-{\phi} \ar[r]^-{\beta} & K \ar[d]^-{I_K} \ar[r] & 0 \\
		\mathcal{E'}(H, K) \equiv  0 \ar[r] & H \ar[r]_-{i'} & P' \ar[r]_-{\beta'} & K \ar[r] & 0 }$$
    is commutative.
\end{enumerate}
\end{definition}
\begin{remark}
    Let $P$ be a commutative Poisson algebra and $End(P)$ be the set of all vector space homomorphisms on $P$. Then $End(P)$ is a Poisson algebra with respect to the binary operation $\cdot$ and $[ \ , \ ]$ defined by $ (f_1 \cdot f_2)(h) = f_1(f_2(h))$ and $[f_1,f_2](h)=f_1 (f_2(h))- f_2 (f_1(h)) $.
\end{remark}
    
\begin{definition} Let $P$ be a commutative Poisson algebra.
\begin{enumerate}
 \item If $P^{\prime}$ is another commutative Poisson algebra with binary operations $\cdot^{\prime}$ and $[\ ,\ ]^{\prime}$. A vector space homomorphism $\phi: P \to P^{\prime}$ is called a Poisson algebra homomorphism if $$\phi(x\cdot y) = \phi(x)\cdot^{\prime}\phi( y) \ \ \text{and} \ \ \ \phi([x, y]) =  [\phi(x),  \phi(y)]^{\prime} \ \ \ \text{for all} \ x,y \in P.$$  
    \item A vector space homomorphism $D : P \to P$ is called a Poisson derivation of  $P$ if  
    \begin{center}
         $D(xy) = D(x)y + xD(y), \ \   D([x, y]) = [D(x), y] + [x, D(y)] \ \ \text{for all}\  x, y \in  P.$ 
    \end{center}
   It is easy to see that the set $Der(P)$ of all Poisson derivations is a Lie subalgebra of $End(P)$. Moreover, $Der(P)$ itself carries the structure of a commutative Poisson algebra under the trivial ring structure. 
\end{enumerate}
\end{definition}
    	
	\section{Extension Theory of Commutative Poisson Algebra}
A fundamental problem in the theory of commutative Poisson algebras is to classify all such algebras $P$ that contain a given ideal $H$ with the property that the quotient $P/H$ is isomorphic to a prescribed algebra $K$, where both $H$ and $K$ are commutative Poisson algebras over the field $\mathbb{F}$. Motivated by the analogy with Lie algebras, we approach this classification problem through the framework of extension theory. This section is devoted to a review of the general extension theory of commutative Poisson algebras.\\
	Consider an extension $ \mathcal{E}(H, K) \equiv	{0}\longrightarrow  H\overset{i} \longrightarrow P\overset{\beta} \longrightarrow  K \longrightarrow {0}$ of $H$ by $K,$  where $H$ and $K$ are commutative Poisson algebras.
   % with $(H, \cdot) \subseteq Z(G, \cdot).$
 Let $t: K \to P$  be a section of $\mathcal{E}(H, K)$. Then, it is easy to see that every element of the group $P$ can be uniquely expressed in the form $h + t(x),$ for some $h\in H$ and $x\in K.$
	
	The group operation $``+"$ and scalar multiplication in $P$ is given by
	\begin{align}
	     \big(h + t(x)\big) + \big(k+t(y)\big)=(h+k) + t(x+y)
	\end{align}
	   \begin{align}
          \alpha\big(h + t(x)\big) =  \alpha h + t(\alpha x)
           \end{align}
	for all $h, k \in H, \ x, y \in K$ and $\alpha \in \mathbb{F}.$
	
	Since $\beta (ht(x)) =  0,$  we have $ ht(x)\in \ker(\beta). $ Thus, for all $x \in K,$ there is a map $\gamma^t_x $ from $H$ to $H$ defined as $\gamma^t_x(h) = ht(x) $. It is easy to see that the map $\gamma^t_x$ is a vector space homomorphism. Also,  $\beta (t(x) t(y)) = x y = \beta(t(x y)),$ we have in consequence $t(x) t(y) - t(x y) \in H. $ Hence, there exists a map $f^t : K \times K \to H$ such that
	\begin{align}
	  f^t(x, y)= t(x) t(y) - t(x y). \label{2.3}   
	\end{align}
	Since $t(0) = 0,$ we have
	\begin{align}
	    f^t(x, 0) = f^t(0, y) =  f^t(0, 0) = 0 \ \text{and} \  f^t(x, y) =  f^t(y, x).\label{2.4}
	\end{align}
Since $t$ is a vector space homomorphism,
	\begin{align*}
		f^t(\alpha x + \beta y, z) &= t(\alpha x + \beta y) t(z) - t\big((\alpha x + \beta y) z\big) \\& = \alpha t(x) t(z) + \beta t(y) t(z) - \alpha t(xz) - \beta t(yz) \\& = \alpha f^t(x, z) + \beta f^t(y, z).
	\end{align*}
	Finally, we have
	\begin{align}
		f^t(\alpha x + \beta y, z) = \alpha f^t(x, z) + \beta f^t(y, z).\label{2.14} 
	\end{align}
	Similarly, 
	\begin{align}
		f^t( x, \alpha y + \beta z) = \alpha f^t(x, y) + \beta f^t(x, z).\label{2.15} 
	\end{align}

	Now, the ring structure $``\cdot"$ in $P$ is given by:
	\begin{align}
		\big(h + t(x)\big) \cdot \big(k+t(y)\big) & = hk + ht(y) + t(x)k + t(x)t(y) \notag \\  \implies  \big(h + t(x)\big) \cdot \big(k+t(y)\big) & = hk + \gamma^t_y(h) + \gamma^t_x(k) + f^t(x, y) + t(xy). \label{2.5}  
	\end{align} 

	Furthermore, $\beta ([t(x), h]) = [\beta (t(x)), \beta(h)] = [x, 0] = 0,$  we have $ [t(x), h] \in \ker(\beta).$ Thus, for all $x \in K,$ there is a map $\Gamma^t_{x} : H \to H$ defined by $\Gamma_{x}^t(h) = [t(x), h].$ It can be easily seen that, $\Gamma^t_{x}$ is a vector space homomorphism on $H$.
	
	Since $\beta$ is a Poisson algebra homomorphism, we have $$\beta (t([x, y])) = [x, y] = [\beta(t(x)), \beta(t(y))] = \beta ([t(x), t(y)]) $$ for all $x, y \in K.$ Thus, $[t(x), t(y)]-t([x, y]) \in H$.  Hence, we have a map $h^t$ from $K \times K$ to $H$ defined as
	\begin{align}
		h^t(x, y) = [t(x), t(y)] - t([x, y]). \label{2.6}
	\end{align}  
	 Since $[t(x), t(x)] = 0 = t([x, x]),$ and $t(0) = 0,$  we have 
	\begin{align}
		h^t(0, x) = h^t(x, 0) = h^t(x, x) = 0. \label{2.7} 
	\end{align} 
Also,
	\begin{align}
		h^t(\alpha x + \beta y, z) &= [t(\alpha x + \beta y), t(z)] - t([\alpha x + \beta y, z]) = \alpha h^t(x, z) + \beta h^t(y, z).\label{2.16}  
	\end{align}
	Similarly,
	\begin{align}
		h^t(x, \alpha y + \beta z) &= \alpha h^t(x, y) + \beta h^t(x, z).\label{2.17}  
	\end{align}
	for all $x, y, z \in K$ and $\alpha, \beta \in \mathbb{F.}$

	Now, the Poisson bracket $``[ \  , \ ]"$ in $P$ is given by
	\begin{align}
		[h+t(x), k+t(y)] & = [h, k] + [h, t(y)] + [t(x), k] + [t(x), t(y)] \notag \\ \implies [h+t(x), k+t(y)] & = [h, k] - \Gamma^t_{y}(h) + \Gamma^t_{x}(k) + h^t(x, y) + t([x, y]). \label{2.8}
	\end{align}   

\begin{lemma}
		$(1)$	The vector space homomorphism $\gamma^t_x : H \to H $ defined by $\gamma^t_x(h) = ht(x)$ satisfies the following identities for all $x,y \in K$ and $h, k\in H:$
        \begin{equation} \label{2.9}
            \begin{rcases}
                \gamma^t_{xy}(h)=\gamma^t_y(\gamma^t_x(h)) - hf^t(x,y)=\gamma^t_x(\gamma^t_y(h)) - hf^t(x,y),\\ \gamma^t_{[x, y]}(h) = \Gamma^t_y(\gamma^t_x(-h)) + \gamma^t_x(\Gamma^t_y(h)) - hh^t(x,y), \\ \gamma_{x}^t(hk)= h\gamma_{x}^t(k)=k \gamma_{x}^t(h),  \\ \gamma_{x}^t([h,k])=[\gamma_{x}^t(h),k] - h  \Gamma_x^t(k).
            \end{rcases}
        \end{equation}

		$(2)$ The vector space homomorphism $\Gamma^t_x : H \to H $ defined by $\Gamma^t_x(h) = [t(x), h]$ is a Poisson algebra derivation. Also, satisfies the following identities for all $x,y \in K$ and $h, k\in H:$
        \begin{equation}\label{2.10}
         \begin{rcases}
           \Gamma^t_{xy}(h) = \gamma^t_y(\Gamma^t_x(h)) + \gamma^t_x(\Gamma^t_y(h)) + [h, f^t(x, y)],\\
            \Gamma ^t_{[x, y]} (h) = \Gamma^t_x \Gamma^t_y(h) - \Gamma^t_y \Gamma^t_x(h) - [h^t(x, y), h].
         \end{rcases}  
        \end{equation}
	\end{lemma}
	\begin{proof}
		$(1)$ Let $x,y \in K$ and $h, k\in H.$ Then 
		\begin{align*} 
		\gamma^t_{xy}(h) &= ht(xy) \\
			&= h\big(t(x)t(y)- f^t(x,y)\big) \\
			&=ht(x)t(y)- hf^t(x,y) \\&=\gamma^t_y(\gamma^t_x(h)) -h f^t(x,y) \\&=\gamma^t_x(\gamma^t_y(h)) - hf^t(x,y)
		\end{align*}
        
		and, 	\begin{align*}
			\gamma^t_{[x, y]}(h) &= ht([x, y]) \\
			&= h\big([t(x), t(y)]- h^t(x,y)\big) \\
			&=h[t(x), t(y)]- hh^t(x,y)\\
			&=[ht(x), t(y)] - [h, t(y)]t(x)- hh^t(x,y)\\
			&= \Gamma^t_y(\gamma^t_x(-h)) + \gamma^t_x(\Gamma^t_y(h)) - hh^t(x,y). 
           \end{align*}
 Also, \begin{align*}
      \gamma_{x}^t(hk) = hkt(x) =h \gamma_{x}^t(k) = k \gamma_{x}^t(h)   \end{align*}
and, \begin{align*}
     \gamma_{x}^t([h,k])=[h,k]t(x)=[h t(x),k] - h [t(x),k]= [\gamma_x^t(h),k]- h \Gamma_x^t(k).
    \end{align*}
            
		$(2)$ Let $h, k\in H$ and $x, y \in K .$ Then
		\begin{align*}
			\hspace{4cm}\Gamma^t_x([h, k]) &= [t(x), [h, k]] \\& = [[t(x), h], k] + [h, [t(x), k]] \ \ \text{(by the Jacobi identity)} \\& = [\Gamma^t_x(h), k] + [h, \Gamma^t_x(k)]. 
		 \end{align*}
		Also, \begin{align*}
			\hspace{-10cm}\Gamma^t_x(hk) &= [t(x), hk] \\& = [t(x), h]k + h[t(x), k] \\& = \Gamma^t_x(h) k + h \Gamma^t_x(k).
		\end{align*}
		Thus, $\Gamma^t_x$ is a Poisson algebra derivation.
		Now, \begin{align*}
			\Gamma^t_{xy}(h)  &= [t(xy), h] = [t(x)t(y) - f^t(x, y), h] = [t(x)t(y), h] + [h, f^t(x, y)] \\&= [t(x), h]t(y) + t(x)[t(y), h]+ [h, f^t(x, y)] \\&= \Gamma^t_{x}(h)t(y) + \Gamma^t_{y}(h) t(x) + [h, f^t(x, y)] \\& = \gamma^t_y(\Gamma^t_{x}(h)) + \gamma^t_x(\Gamma^t_{y}(h)) + [h, f^t(x, y)].  
		\end{align*}
		Thus, 
		\begin{align*}
		 \Gamma^t_{xy}(h)  = \gamma^t_y(\Gamma^t_{x}(h)) + \gamma^t_x(\Gamma^t_{y}(h)) + [h, f^t(x, y)]. 
		\end{align*}
        Since $[[t(x), t(y)], h] + [[t(y), h], t(x)] + [[h, t(x)], t(y)] = 0,$ we have
	\begin{align*}
		 \Gamma^t_{[x, y]} (h) = \Gamma^t_x \Gamma^t_y(h) - \Gamma^t_y \Gamma^t_x(h) - [h^t(x, y), h].   
	\end{align*}  
	\end{proof}
    
	\begin{lemma} \label{a}
		Let $H$ and $K$ be two commutative Poisson algebras. Then
	\end{lemma}
	\begin{enumerate}
		\item 	The map $\gamma^t$ from  $K$  to $End(H)$ defined by $\gamma^t(x) = \gamma^t_{x} $ is a vector space homomorphism. 
		\item 	The map $\Gamma^t: K \to  Der(H)$ defined by $\Gamma^t(x) = \Gamma^t_{x}$ is a vector space homomorphism.
	\end{enumerate}
	
	\begin{proof}
		
	$(1)$ Let $x, y \in K$, $h\in H,$ and $\alpha,\beta \in \mathbb{F}.$  Then
    \begin{align*}
       \gamma^t_{\alpha x+ \beta y}(h) &= ht(\alpha x+ \beta y) = ht(\alpha x) + ht(\beta y) \\&= \alpha ht(x)+ \beta ht(y) \\&=  \alpha \gamma^t_{x}(h) + \beta \gamma^t_{y}(h),  
    \end{align*}
    since $t$ is a vector space homomorphism. Hence, $$ \gamma^t{(\alpha x+ \beta y)} = \alpha \gamma^t(x) + \beta \gamma^t(y). $$  This completes the proof.\\
    $(2)$ Let $x, y \in K$, $h\in H$ and $\alpha, \beta \in \mathbb{F}.$ Then
    \begin{align*}
     \Gamma^t_{\alpha x+\beta y}(h) & = [t(\alpha x+\beta y), h] = [t(\alpha x) + t(\beta y), h] \\& = \alpha [t(x) , h]+\beta [ t(y), h] \\&=  \alpha \Gamma^t_{x}(h) + \beta \Gamma^t_{y}(h).   
    \end{align*}

Hence, $ \Gamma^t{(\alpha x+ \beta y)} = \alpha \Gamma^t(x) + \beta \Gamma^t(y). $ This completes the proof.
\end{proof}
	Now, we will see the properties of the function $f^t$ and $h^t$ by using equation $(\ref{2.5})$ and $ (\ref{2.8}).$ \\
	Consider the expression
	\begin{align}
		\big(t(x)\cdot t(y)\big) \cdot t(z) &= \big(f^{t}(x,y)+ t(x y)\big) \cdot t(z) \notag\\&= \gamma^t_z ( f^{t}(x,y)) + t(x y)\cdot t(z) \notag\\&= \gamma^t_z( f^{t}(x,y)) + f^{t}(x y, z)+ t((x y) z).\label{2.11} 
	\end{align}
	On the other hand,
	\begin{align}
		 t(x)\cdot \big(t(y) \cdot t(z)\big) &= t(x) \cdot  \big(f^{t}(y,z)  + t(y z)\big) \notag\\&=  \gamma_x^t(f^{t}(y,z)) + t(x)\cdot t(y z) \notag\\&=  \gamma_x^{t}(f^{t}(y,z)) + f^{t}(x,y z)+ t(x(y z)). \label{2.12} 
	\end{align}
	By equating equation $(\ref{2.11})$ and $(\ref{2.12}),$ we have  
	\begin{align}
		\gamma_z^t( f^{t}(x,y)) +   f^{t}(x y, z) =  \gamma_x^{t}(f^{t}(y,z)) + f^{t}(x,y z). \label{2.13} 
	\end{align}

	Now, consider the Jacobi identity 
	\begin{align*}
		[[t(x), t(y)], t(z)] + [[t(y), t(z)], t(x)] + [[t(z), t(x)], t(y)] = 0.
	\end{align*}This implies that,
	\begin{align*}
		t\big([[x, y], z]\big) + t\big([[y, z], x]\big) + t\big([[z, x], y]\big) + h^t([x,y], z) - \Gamma^t_z(h^t(x, y)) + h^t([y, z], x) \\& \hspace*{-13cm} - \Gamma^t_x(h^t(y, z)) + h^t([z,x], y) - \Gamma^t_y(h^t(z, x)) =  0.
	\end{align*}
	Since $t$ is a vector space homomorphism, we have
	\begin{align}
		h^t([x,y], z) + h^t([y, z], x) + h^t([z,x], y) = \Gamma^t_z(h^t(x, y)) + \Gamma^t_x(h^t(y, z)) + \Gamma^t_y(h^t(z, x)).\label{2.18}    
	\end{align}

	Now, consider the expression
	\begin{align}
		[t(x)\cdot t(y), t(z)] &= [f^{t}(x,y)+ t(x y), t(z)] \notag\\& =  [t(x y), t(z)]  - \Gamma_z^t( f^{t}(x,y))  \notag\\&=  h^t(xy,z)+  t([xy,z]) -\Gamma_z^t(f^t(x,y)).\label{2.19} 
	\end{align}
	On the other hand,
	\begin{align}
		[t(x)\cdot t(y), t(z)] & = [t(x),t(z)] \cdot t(y) + t(x) \cdot [t(y),t(z)] \notag\\&= \big(h^t(x,z) + t([x,z])\big)  \cdot t(y) + t(x) \cdot \big(h^t(y,z) + t([y,z])\big) \notag\\&=  \gamma_y^t(h^t(x,z)) +   f^t([x,z],y) + t([x,z]y)  + \gamma_x^t(h^t(y,z)) \notag\\&
        \quad + f^t(x,[y,z]) + t(x[y,z]).\label{2.20} 
	\end{align}
	Now, equating equation $(\ref{2.19})$ and $(\ref{2.20}),$ we have 
	\begin{align}
		h^t(xy,z) & =  \Gamma_z^t(f^{t}(x,y)) +\gamma_y^t( h^t(x,z)) + f^t([x,z],y) + \gamma_x^t(h^t(y,z))+ f^t(x,[y,z]). \label{2.21}  
	\end{align}

	\begin{definition} \label{Fac}
		Let $H$ and $K$ be two commutative Poisson algebras. Then a factor system is a sextuple $(H, K, f, h, \gamma, \Gamma),$ where $f: K \times K \to H $ is a symmetric bilinear map and  $ h: K \times K \to H $ is a alternating map, and $\gamma: K \to End(H)$ and $\Gamma: K \to Der(H)$  are vector space homomorphisms satisfying equations $(\ref{2.9})$, $(\ref{2.10})$,  $(\ref{2.13})$, $(\ref{2.18})$, and $(\ref{2.21})$. 
	\end{definition}
	\begin{remark}
	    From the above discussion, every commutative Poisson algebra extension $\mathcal{E}(H, K)$ with a choice of a section $t,$ we have a factor system $Fac(\mathcal{E}, t) = (H, K,  f^t, h^t, $ $ \gamma^t, \Gamma^t) $ called as factor system given by the commutative Poisson algebra extension $\mathcal{E}(H, K).$
	\end{remark}
Now, we prove the following proposition:  
	
	\begin{proposition}\label{P2.3}
		Every extension $\mathcal{E}(H, K)$ of a commutative Poisson algebra with a choice of a section $t$ determines a factor system $Fac(\mathcal{E}, t) = (H, K,  f^t, h^t, \gamma^t, \Gamma^t), $ and conversely, for a given factor system $(H, K,  f, h, \gamma, \Gamma),$ there exists an extension $\mathcal{E}(H, K)$ of $H$ by $K$ with a section $t$ such that  $(H, K, f, h, \gamma, \Gamma) = Fac(\mathcal{E}, t).$      
	\end{proposition}
	\begin{proof}
		From the above discussions, it can be seen that every extension $\mathcal{E}(H, K)$ of $H$ by $K$ with a given section $t$ determines a factor system $ Fac(\mathcal{E}, t) = (H, K,  f^t, h^t, \gamma^t, \Gamma^t).$\\ 
		Conversely, let $P = H \times K=\{(h,x) : h\in H, x \in K\}$. It can be easily verified that $P$ forms a commutative Poisson algebra associated with the following operations
        $$\alpha(h,x)=(\alpha h,\alpha x),$$
        $$(h,x)+ (k,y)=(h+k, x+y),$$
		$$(h, x)  \cdot (k, y) = \big(hk + \gamma_y(h) + \gamma_x(k) + f(x, y),\  xy\big),$$ 
		$$[(h, x), (k, y)]  =  \big([h, k] - \Gamma_{y}(h) + \Gamma_{x}(k) + h(x, y), \ [x, y]\big),$$ where $\alpha \in \F, h,k\in H,$ and $x,y \in K$. \\
        Also, $$ \mathcal{E}(H, K) \equiv	{0}\longrightarrow  H\overset{i} \longrightarrow P\overset{p} \longrightarrow  K \longrightarrow {0}$$ is an extension of $H$ by $K,$ where $i(h) = (h, 0) $ and $p(h, x) = x.$ Let $t$ be a section of $\mathcal{E}$ given by $t(x) = (0, x).$ By an easy computation, it can be seen that $(H, K,  f, h, \gamma, \Gamma) =  (H, K,  f^t, h^t, \gamma^t, \Gamma^t) $. 
	\end{proof}

\subsection{Equivalence between category EXT of extensions and category of FAC of factor systems} \ \ \

Let $(\lambda, \mu, \nu )$ be a morphism from the extension $\mathcal{E}(H_1,K_1)$ to the extension $\mathcal{E}(H_2,K_2)$. Then we have the following commutative diagram:

	$$ 	
	\xymatrix{\mathcal{E}(H_1, K_1) \equiv 0 \ar[r] & H_1 \ar[d]_-{\lambda} \ar[r]^-{i_1} & P_1 \ar[d]_-{\mu} \ar[r]^-{p_1} & K_1 \ar[d]^-{\nu} \ar[r] & 0 \\
		\mathcal{E}(H_2, K_2) \equiv  0 \ar[r] & H_2 \ar[r]_-{i_2} & P_2 \ar[r]_-{p_2} & K_2 \ar[r] & 0
	}$$
	
	Let $t_1$ and $t_2$ be sections of $\mathcal{E}(H_1, K_1)$ and $\mathcal{E}(H_2, K_2),$ respectively. Consider the corresponding factor systems $(H_1, K_1, f^{t_1}, h^{t_1}, \gamma^{t_1}, \Gamma^{t_1})$ and $(H_2, K_2, f^{t_2}, h^{t_2}, \gamma^{t_2}, \Gamma^{t_2})$ of $\mathcal{E}(H_1, K_1)$ and $\mathcal{E}(H_2, K_2),$ respectively. By the commutativity of the diagram, for each 
	$x\in K_1$ we have
	$$p_2(\mu (t_1(x)))=\nu (p_1 (t_1 (x))) = \nu (x) = p_2 (t_2(\nu (x))). $$  It follows that there exists a unique element $g(x) \in H_2$ satisfying
	\begin{align}
		\mu(t_1(x))=t_2(\nu (x)) + g(x) \label{2.22} 
	\end{align}
	for all $x \in K_1$. This defines a map $g$ from $K_1$ to $H_2$. Since $\mu, \nu, t_1, t_2$ are vector space homomorphisms, it follows that $g$ is a vector space homomorphism.
	Now, observe that
	\begin{align}
		\mu([t_1(x),t_1(y)]) &= \mu(t_1([x,y]) + h^{t_1}(x,y)) \notag\\
		&= \mu(t_1([x,y])) + \mu(h^{t_1}(x,y)) \notag\\
		&= t_2(\nu([x,y])) + g([x,y]) + \lambda(h^{t_1}(x,y)).\label{2.23}
	\end{align}
	Since $\mu$ is a Poisson algebra homomorphism, we have
	\begin{align}
		\mu([t_1(x),t_1(y)]) &= [\mu(t_1(x)),\mu(t_1(y))] \notag\\
		&\hspace{-2.1cm}= [t_2(\nu (x)) + g(x),t_2(\nu (y)) + g(y)] \notag\\
		&\hspace{-2.1cm}= [t_2(\nu (x)),t_2(\nu (y))] +[t_2(\nu (x)) ,g(y)] +[g(x),t_2(\nu (y))] +[g(x),g(y)] \notag\\
		&\hspace{-2.1cm}= t_2([\nu (x),\nu (y)]) 
		+ h^{t_2}(\nu(x),\nu(y)) + \Gamma^{t_2}_{\nu(x)}(g(y)) -\Gamma^{t_2}_{\nu(y)}(g(x))  + [g(x),g(y)]. \label{2.24}
	\end{align}
	Since $\nu$ is a Poisson algebra homomorphism, by comparing equations $(\ref{2.23})$ and $(\ref{2.24}),$ we obtain
	\begin{align}
		h^{t_2}(\nu(x),\nu(y)) + \Gamma^{t_2}_{\nu(x)}(g(y)) -\Gamma^{t_2}_{\nu(y)}(g(x)) + [g(x),g(y)] \notag\\& \hspace{-6.5cm} =g([x,y]) + \lambda(h^{t_1}(x,y)) \label{2.25}
	\end{align} for all $x,y \in K_1$. Further,
	\begin{align*}
		\lambda(\Gamma^{t_1}_{x}(h)) &= \lambda([t_1(x),h])\\ &= \mu([t_1(x),h]) \\
		&= [t_2(\nu (x)) + g(x),\mu(h)] \\
		&= [t_2(\nu (x)),\lambda(h)] + [g(x),\lambda(h)] \\
		&= \Gamma^{t_2}_{\nu(x)}(\lambda(h)) + [g(x),\lambda(h)],
	\end{align*}
	so that 
	\begin{align}
		\lambda(\Gamma^{t_1}_{x}(h))=\Gamma^{t_2}_{\nu(x)}(\lambda(h)) + [g(x),\lambda(h)]. \label{2.26}
	\end{align} 
	Now, 
	\begin{align}
		\mu(t_1(x) t_1(y)) &= \mu(f^{t_1}(x, y) + t_1(x y)) \notag\\
		&= \lambda(f^{t_1}(x,y)) + t_2(\nu (xy)) + g(xy). \label{2.27}
	\end{align}
	On the other hand, since $\mu$ is a Poisson algebra homomorphism. So, 
	\begin{align}
		\mu(t_1(x) t_1(y)) &= \mu(t_1(x))\cdot \mu(t_1(y)) \notag\\
		&\hspace{-1cm}= \big(g(x) + t_2(\nu (x))\big) \cdot \big(g(y) + t_2(\nu (y))\big) \notag\\
		&\hspace{-1cm}= g(x) \cdot g(y) + \gamma^{t_2}_{\nu(y)}(g(x)) + \gamma^{t_2}_{\nu(x)}(g(y)) + f^{t_2}(\nu(x),\nu(y)) + t_2(\nu(x)\nu(y)). \label{2.28}
	\end{align}
	Since $\nu$ is a Poisson algebra homomorphism, by comparing equations $(\ref{2.27})$ and $(\ref{2.28}),$ we get
	\begin{align}
		\lambda(f^{t_1}(x,y))+ g(xy) &= g(x)\cdot g(y) + \gamma^{t_2}_{\nu(y)}(g(x)) + \gamma^{t_2}_{\nu(x)}(g(y)) + f^{t_2}(\nu(x),\nu(y)). \label{2.29}
	\end{align}
	Next,
	\begin{align*}
		\lambda(\gamma^{t_1}_{x}(h)) &=
		\lambda(h \cdot t_1(x))\\ &= \lambda(h) \cdot \lambda(t_1(x)) \\
		&= \lambda(h) \cdot \mu(t_1(x))  \\
		&=\lambda(h) \cdot  \big(g(x) + t_2(\nu (x))\big) \\
		&= \lambda(h) \cdot g(x) + \gamma^{t_2}_{\nu(x)}(\lambda(h)). 
	\end{align*}
	Thus,
	\begin{align}
		\lambda(\gamma^{t_1}_{x}(h)) &=  \lambda(h) \cdot g(x) + \gamma^{t_2}_{\nu(x)}(\lambda(h)). \label{2.30}
	\end{align}
	Now, since $\mu$ is a Poisson algebra homomorphism. So,
	\begin{align}
		&\mu([t_1(x)\cdot t_1(y),t_1(z)]) = [\mu(t_1(x))\cdot \mu(t_1(y)),\mu(t_1(z))] \notag\\
		&= [\mu(t_1(x)),\mu(t_1(z))] \cdot \mu(t_1(y)) + \mu(t_1(x))\cdot [\mu(t_1(y)),\mu(t_1(z))] \notag\\
	 	&= \big(t_2([\nu (x),\nu (z)]) + h^{t_2}(\nu(x),\nu(z)) + \Gamma^{t_2}_{\nu(x)}(g(z)) -\Gamma^{t_2}_{\nu(z)}(g(x))  + [g(x),g(z)]\big) \notag\\& \quad \cdot \big(g(y) + t_2(\nu (y))\big) + \big(g(x) + t_2(\nu (x))\big) \cdot \big(t_2([\nu (y),\nu (z)])  + h^{t_2}(\nu(y),\nu(z)) \notag\\& \quad + \Gamma^{t_2}_{\nu(y)}(g(z)) -\Gamma^{t_2}_{\nu(z)}(g(y)) + [g(y),g(z)]\big) \notag\\
		&= \big(h^{t_2}(\nu(x),\nu(z)) + \Gamma^{t_2}_{\nu(x)}(g(z)) -\Gamma^{t_2}_{\nu(z)}(g(x))  + [g(x),g(z)]\big) \cdot g(y)  + \gamma^{t_2}_{\nu(y)}\big(h^{t_2}(\nu(x),\nu(z))  \notag\\& \quad + \Gamma^{t_2}_{\nu(x)}(g(z)) -\Gamma^{t_2}_{\nu(z)}(g(x)) + [g(x),g(z)]\big) 
		  + \gamma^{t_2}_{\nu([x,z])}(g(y)) + f^{t_2}(\nu([x,z]),\nu(y))  \notag\\& \quad  + g(x) \cdot 
	 	\big(h^{t_2}(\nu(y),\nu(z))  + \Gamma^{t_2}_{\nu(y)}(g(z)) -\Gamma^{t_2}_{\nu(z)}(g(y)) + [g(y),g(z)]\big)  \notag\\& \quad+ \gamma^{t_2}_{\nu([y,z])}(g(x)) 
		+ \gamma^{t_2}_{\nu(x)}\big(h^{t_2}(\nu(y),\nu(z))   + \Gamma^{t_2}_{\nu(y)}(g(z)) -\Gamma^{t_2}_{\nu(z)}(g(y)) + [g(y),g(z)]\big)  \notag\\& \quad+ 
		 f^{t_2}(\nu(x),\nu([y,z])) +t_2\big(\nu(x)\nu([y,z])\big)+ t_2\big(\nu([x,z])\nu(y)\big). \label{2.31}
	 \end{align}
On the other hand
	\begin{align}
		\mu\big([t_1(x)\cdot t_1(y),t_1(z)]\big) 
		&= \mu\big([f^{t_1}(x,y) + t_1(xy), t_1(z)]\big) \notag\\
		&\hspace{-1.7cm}= \mu([f^{t_1}(x,y),t_1(z)]) + \mu([t_1(xy), t_1(z)]) \notag\\
		&\hspace{-1.7cm}= -\mu(\Gamma^{t_1}_{z}(f^{t_1}(x,y))) + \mu\big(t_1([xy,z]) + h^{t_1}(xy,z)\big) \notag\\
		&\hspace{-1.7cm}= -\mu(\Gamma^{t_1}_{z}(f^{t_1}(x,y))) + \mu(t_1([xy,z])) + \mu(h^{t_1}(xy,z)) \notag\\
		&\hspace{-1.7cm}=  -\lambda(\Gamma^{t_1}_{z}(f^{t_1}(x,y))) + t_2(\nu([xy,z])) + g([xy,z]) + \lambda(h^{t_1}(xy,z)).\label{2.32} 
	\end{align}
	Since $\nu$ is a Poisson algebra homomorphism, by comparing equations  $(\ref{2.31})$ and $(\ref{2.32}),$ we get 
\begin{align}
		&\hspace{-0.6cm}g([xy,z]) -\lambda\big(\Gamma^{t_1}_{z}(f^{t_1}(x,y))\big) + \lambda(h^{t_1}(xy,z)) \notag\\&\hspace{-0.4cm}= \big(h^{t_2}(\nu(x),\nu(z)) + \Gamma^{t_2}_{\nu(x)}(g(z)) -\Gamma^{t_2}_{\nu(z)}(g(x))  + [g(x),g(z)]\big) \cdot g(y)  + \gamma^{t_2}_{\nu(y)}\big(h^{t_2}(\nu(x),\nu(z))  \notag\\& \quad + \Gamma^{t_2}_{\nu(x)}(g(z)) -\Gamma^{t_2}_{\nu(z)}(g(x)) + [g(x),g(z)]\big) 
		  + \gamma^{t_2}_{\nu([x,z])}(g(y)) + f^{t_2}(\nu([x,z]),\nu(y))  \notag\\& \quad  + g(x) \cdot 
	 	\big(h^{t_2}(\nu(y),\nu(z))  + \Gamma^{t_2}_{\nu(y)}(g(z)) -\Gamma^{t_2}_{\nu(z)}(g(y)) + [g(y),g(z)]\big)  \notag\\& \quad+ \gamma^{t_2}_{\nu([y,z])}(g(x)) 
		+ \gamma^{t_2}_{\nu(x)}\big(h^{t_2}(\nu(y),\nu(z))   + \Gamma^{t_2}_{\nu(y)}(g(z)) -\Gamma^{t_2}_{\nu(z)}(g(y)) + [g(y),g(z)]\big)  \notag\\& \quad+ 
		 f^{t_2}(\nu(x),\nu([y,z])). \label{2.36}
	\end{align}
	
	Thus a morphism  $(\lambda, \mu, \nu )$ between extensions $\mathcal{E}(H_1,K_1)$ and $\mathcal{E}(H_2,K_2),$ together with  choices of sections $t_1$ and $t_2$ for these extensions, naturally induces a vector space homomorphism $g$ from $K_1$ to $H_2$ such that the triple $(\lambda, g, \nu )$ satisfies $(\ref{2.25})$,  $(\ref{2.26})$,  $(\ref{2.29})$,  $(\ref{2.30})$, and  $(\ref{2.36})$. It can be seen as a morphism from the factor system $(H_1, K_1, f^{t_1}, h^{t_1}, \gamma^{t_1}, \Gamma^{t_1})$ to $(H_2, K_2, f^{t_2}, h^{t_2}, \gamma^{t_2}, \Gamma^{t_2})$.
	\begin{definition}
 Let $t_1$ and $ t_2$ be two sections of extension $\mathcal{E}(H,K)$, and let $(H, K, f^{t_1}, h^{t_1}$ $ ,\gamma^{t_1}, \Gamma^{t_1})$ and $(H, K, f^{t_2}, h^{t_2}, \gamma^{t_2}, \Gamma^{t_2})$ be the corresponding factor systems. A morphism  $(\lambda, g, \nu )$ from the factor system $(H, K, f^{t_1}, h^{t_1}, \gamma^{t_1}, \Gamma^{t_1})$ to $(H, K, f^{t_2}, h^{t_2}, \gamma^{t_2}, \Gamma^{t_2})$ is said to be an equivalence between factor systems if $\lambda=I_H$ and $\nu =I_K$.
\end{definition}
Let $(\lambda_1, \mu_1, \nu_1 )$ be a morphism from an extension  $$ \mathcal{E}(H_1, K_1) \equiv	{0}\longrightarrow  H_1\overset{i_1} \longrightarrow P_1\overset{p_1} \longrightarrow  K_1 \longrightarrow {0}$$ to an extension  $$ \mathcal{E}(H_2, K_2) \equiv	{0}\longrightarrow  H_2\overset{i_2} \longrightarrow P_2\overset{p_2} \longrightarrow  K_2 \longrightarrow {0,}$$ and $(\lambda_2, \mu_2, \nu_2 )$ be another morphism from the extension $\mathcal{E}(H_2, K_2)$ to the extension  $$ \mathcal{E}(H_3, K_3) \equiv	{0}\longrightarrow  H_3\overset{i_3} \longrightarrow P_3\overset{p_3} \longrightarrow  K_3 \longrightarrow {0.}$$ Let $t_1, t_2 $ and $t_3$ be corresponding choices of the sections. Then as in $(\ref{2.22}),$ $$\mu_1(t_1(x))=t_2(\nu_1 (x)) + g_1(x)  $$ and   $$\mu_2(t_2(x))=t_3(\nu_2 (x)) + g_2(x),  $$ where $g_1$ is the uniquely determined vector space homomorphism from $K_1$ to $H_2,$ and $g_2$ is that from $K_2$ to $H_3.$ In turn, we have 
\begin{align*}
\mu_2(\mu_1(t_1(x))) &= \mu_2\big(t_2(\nu_1(x)) + g_1(x)\big) =  \mu_2(t_2(\nu_1(x))) + \mu_2(g_1(x)) \\& =  t_3(\nu_2(\nu_1(x))) + g_2(\nu_1(x)) + \lambda_2(g_1(x)) = t_3\big((\nu_2 \circ \nu_1)(x)\big) + g_3(x),
\end{align*}
	Where $g_3(x) = g_2(\nu_1(x)) + \lambda_2(g_1(x)), $ for each $x\in K_1.$ It follows that the composition $(\lambda_2 \circ \lambda_1, \ \mu_2 \circ \mu_1, \ \nu_2 \circ \nu_1)$ induces the triple  $( \lambda_2 \circ \lambda_1 , \ g_3, \ \nu_2 \circ \nu_1).$
	
	Now we introduce the category $FAC$ whose objects are factor systems, and a morphism from $(H_1, K_1,  f^{1}, h^{1}, \gamma^{1}, \Gamma^{1})$ to $(  H_2, K_2, f^{2}, h^{2}, \gamma^{2}, \Gamma^{2})$ is a triple $(\lambda, g, \nu),$ where  $\lambda : H_1 \to H_2$, $\nu : K_1 \to K_2$ are Poisson algebra homomorphism, and $g$ is a vector space homomorphism from $K_1$ to $H_2$ satisfying  $(\ref{2.25})$,  $(\ref{2.26})$,  $(\ref{2.29})$,  $(\ref{2.30})$, and $(\ref{2.36})$ with replaced by $( \gamma^{1}, f^{1}, h^{1}, \Gamma^{1})$ to $( f^{t_1}, h^{t_1}, \gamma^{t_1}, \Gamma^{t_1})$ and $( f^{2}, h^{2}, \gamma^{2}, \Gamma^{2})$  to  $( f^{t_2}, h^{t_2}, \gamma^{t_2}, \Gamma^{t_2})$. 
	
	The following theorem is a consequence of the above discussion.
	
	\begin{theorem}
		There is an equivalence between the category $EXT$ of extensions to the category $FAC$ of factor systems. 
	\end{theorem}
\section{Ring center extension and cohomology}
In this section, we discuss the theory of ring center extension in a similar manner to Section $3.$ Throughout this section, we assume that $Der(P)$ is a commutative Poisson algebra with respect to the trivial ring structure. 
\begin{lemma}
 Let $P$ be a commutative Poisson algebra. The set $Z_{R}(P) = \{ x\in P \ | \ xy = 0, \ \text{for all }y \in P  \}$ is an ideal of $P,$ called a ring center of $P.$  
\end{lemma}
\begin{proof}
 It is clear that $Z_{R}(P)$ is a subspace of $P.$ By definition, $ xy \in Z_R(P)$ for all $y\in P$ and $ x\in Z_R(P).$ Now, let $x\in Z_R(P)$ and $y, z \in P.$ Then, by Definition \ref{D1.3}, we have $$ [x, y] z = [x z, y] - x [z,y]  = 0.$$ Hence, $[x, y] \in Z_R(P).$  This completes the proof. 
\end{proof}

\begin{definition}
   An extension $\mathcal{E}(H, K) \equiv	{0}\longrightarrow  H\overset{i} \longrightarrow P\overset{\beta} \longrightarrow  K \longrightarrow {0}$ of a commutative Poisson algebra $H$ with trivial ring structure by a commutative Poisson algebra $K$, is called a ring center extension if $H$ is contained in the ring center $Z_R(P)$ of $P$.   
\end{definition}

Let $ \mathcal{E}(H, K) \equiv	{0}\longrightarrow  H\overset{i} \longrightarrow P\overset{\beta} \longrightarrow  K \longrightarrow {0}$ be a ring center extension of $H$ by $K$ and $t: K \to P$ be a section of $\mathcal{E}(H, K)$. Since $H\subseteq Z_R(P)$, that is, $hx = 0 $ for all $h\in H$ and $x\in P.$
Now, by section $3,$ the ring structure $\cdot$ and Poisson bracket $[ \  , \ ]$ in $P$ is defined by    

\begin{align} \label{3.1}
	  \big(h + t(x)\big) \cdot \big(k+t(y)\big) = f^t(x, y) + t(xy),
    \end{align} 
    and 
\begin{align} \label{3.2}
		[h+t(x), k+t(y)] =  [h, k] - \Gamma^t_{y}(h) + \Gamma^t_{x}(k) + h^t(x, y) + t([x, y]), 
\end{align}
where $f^t, h^t: K\times K \to H$ are symmetric bilinear and alternating maps, respectively, and $\Gamma^t_{x}:  H \to H$ is a Poisson algebra derivation  defined as in Section $3,$ and satisfying the following conditions:
\begin{equation} \label{3.3} 
            \begin{rcases}
               \hspace{-0.376cm}(1) \ \Gamma^t_{xy}(h) = [h, f^t(x, y)]\\ \hspace{-0.376cm} (2)  \ \Gamma ^t_{[x, y]} (h) = \Gamma^t_x \Gamma^t_y(h) - \Gamma^t_y \Gamma^t_x(h) - [h^t(x, y), h] \\  \hspace{-0.376cm} (3) \ f^{t}(x y, z) = f^{t}(x,y z)  \\ \hspace{-0.376cm}(4)  \  h^t([x,y], z) + h^t([y, z], x) + h^t([z,x], y)=\Gamma^t_z(h^t(x, y)) + \Gamma^t_x(h^t(y, z)) + \Gamma^t_y(h^t(z, x)) \\ 	\hspace{-0.376cm} (5) \	h^t(xy,z)  =  \Gamma_z^t(f^{t}(x,y))  + f^t([x,z],y) + f^t(x,[y,z]). 
            \end{rcases}
        \end{equation}

	\begin{definition} \label{D3.2}
		Let $H$ be a commutative Poisson algebra with trivial ring structure and $K$ be a commutative Poisson algebra. Then a ring center factor system is a quintuple $(H, K, f, h, \Gamma),$ where $f: K \times K \to H $ is a symmetric bilinear map and  $ h: K \times K \to H $ is an alternating map, and $\Gamma: K \to Der(H)$  is a vector space homomorphism satisfying the equation $(\ref{3.3})$. 
	\end{definition}
	So, we can say that for every ring center extension $\mathcal{E}(H, K)$ with a choice of a section $t,$ we have a ring center factor system $RFac(\mathcal{E}, t) = (H, K,  f^t, h^t, \Gamma^t), $ called as ring center factor system given by the ring center extension $\mathcal{E}(H, K).$ Now, we have the following proposition:  

\begin{proposition}\label{P3.2}
		Every ring center extension $\mathcal{E}(H, K)$ with a choice of a section $t$ determines a ring center factor system $RFac(\mathcal{E}, t) = (H, K,  f^t, h^t, \Gamma^t), $ and conversely, for a given ring center factor system $(H, K,  f, h, \Gamma),$ there exists a ring center extension $\mathcal{E}(H, K)$ of $H$ by $K$ with a section $t$ such that  $(H, K, f, h, \Gamma) = RFac(\mathcal{E}, t).$      
	\end{proposition}

\subsection{Equivalence between category REXT of ring center extensions and category RFAC of ring center factor systems} \ \ \ 

Let $(\lambda, \mu, \nu )$ be a morphism from the ring center extension $\mathcal{E}(H_1,K_1)$ to the ring center extension $\mathcal{E}(H_2,K_2)$. Let $t_1$ and $t_2$ be sections of $\mathcal{E}(H_1, K_1)$ and $\mathcal{E}(H_2, K_2),$ respectively. Consider the corresponding factor systems $(H_1, K_1, f^{t_1}, h^{t_1}, \Gamma^{t_1})$ and $(H_2, K_2, f^{t_2},$ $ h^{t_2}, \Gamma^{t_2})$ of $\mathcal{E}(H_1, K_1)$ and $\mathcal{E}(H_2, K_2),$ respectively.

\begin{definition}
    A morphism between two ring center factor systems $(H_1, K_1, f^{t_1}, h^{t_1}, \Gamma^{t_1})$ and $(H_2, K_2, f^{t_2}, h^{t_2}, \Gamma^{t_2})$ is a triple $(\lambda, g, \nu)$, where  $\lambda: H_1 \to H_2$, $\nu: K_1 \to K_2$ are Poisson algebra homomorphisms, and $g: K_1 \to H_2$ is a vector space homomorphism satisfying the following conditions
    \begin{equation} \label{x}
        \begin{rcases}
        \hspace{-0.44cm}  (1) \ h^{t_2}(\nu(x),\nu(y)) + \Gamma^{t_2}_{\nu(x)}(g(y)) -\Gamma^{t_2}_{\nu(y)}(g(x)) + [g(x),g(y)]  =g([x,y])   + \lambda(h^{t_1}(x,y)) \\
        \hspace{-0.44cm}(2) \ \lambda(\Gamma^{t_1}_{x}(h))=\Gamma^{t_2}_{\nu(x)}(\lambda(h)) + [g(x),\lambda(h)] \\
       \hspace{-0.44cm} (3) \ \lambda(f^{t_1}(x,y))+ g(xy) = f^{t_2}(\nu(x),\nu(y)) \\
        \hspace{-0.44cm}(4) \ g([xy,z]) -\lambda(\Gamma^{t_1}_{z}(f^{t_1}(x,y)) - h^{t_1}(xy,z))  = f^{t_2}(\nu([x,z]),\nu(y)) + f^{t_2}(\nu(x),\nu([y,z]))
        \end{rcases}
    \end{equation}
\end{definition}

\begin{definition}
Let $t_1$ and $t_2$ be two sections of the ring center extension $\mathcal{E}(H,K)$, and let $(H, K, f^{t_1}, h^{t_1}, \Gamma^{t_1})$ and $(H, K, f^{t_2}, h^{t_2}, \Gamma^{t_2})$ be the corresponding ring center factor systems. A morphism  $(\lambda, g, \nu )$ from the ring center factor system $(H, K, f^{t_1}, h^{t_1}, \Gamma^{t_1})$ to $(H, K, f^{t_2}, h^{t_2}, \Gamma^{t_2})$ is said to be an equivalence between the ring center factor systems if $\lambda=I_H$ and $\nu =I_K$.
\end{definition}

	Let $(\lambda_1, \mu_1, \nu_1 )$ be a morphism from a ring center extension $\mathcal{E}(H_1, K_1)$   to a ring center extension $\mathcal{E}(H_2, K_2),$ and $(\lambda_2, \mu_2, \nu_2 )$ be another morphism from a ring center extension $\mathcal{E}(H_2, K_2)$ to a ring center extension $ \mathcal{E}(H_3, K_3).$ Let $t_1, t_2 $ and $t_3$ be corresponding choices of sections. Then, we have vector space homomorphisms $ g_1: K_1 \to H_2,$ and $g_2: K_2 \to H_3$ (defined as in equation \ref{2.22}). In turn, we have 
	\begin{align*}
		\mu_2(\mu_1(t_1(x))) &= \mu_2\big(t_2(\nu_1(x)) + g_1(x)\big) =  \mu_2(t_2(\nu_1(x))) + \mu_2(g_1(x)) \\& =  t_3(\nu_2(\nu_1(x))) + g_2(\nu_1(x)) + \lambda_2(g_1(x)) = t_3\big((\nu_2 \circ \nu_1)(x)\big) + g_3(x),
	\end{align*}
	where $g_3(x) = g_2(\nu_1(x)) + \lambda_2(g_1(x)), $ for each $x\in K_1.$ It follows that the composition $(\lambda_2 \circ \lambda_1, \ \mu_2 \circ \mu_1, \ \nu_2 \circ \nu_1)$ induces the triple $( \lambda_2 \circ \lambda_1 , \ g_3, \ \nu_2 \circ \nu_1 ).$
	
	Now, we introduce the category $RFAC$ whose objects are ring center factor systems, and the morphism from $(  H_1, K_1, f^{1}, h^{1}, \Gamma^{1})$ to $(  H_2, K_2, f^{2}, h^{2}, \Gamma^{2})$ is a triple $(\lambda, g, \nu),$ where $\nu: K_1 \to K_2,$ $\lambda: H_1 \to H_2$ are Poisson algebra homomorphisms, and $g$ is a vector space homomorphism from $K_1$ to $H_2$ satisfying equation $(\ref{x})$ with replaced by $(f^{1}, h^{1}, \Gamma^{1})$ to $( f^{t_1}, h^{t_1}, \Gamma^{t_1})$ and $( f^{2}, h^{2}, \Gamma^{2})$  to  $( f^{t_2}, h^{t_2}, \Gamma^{t_2})$. 
	
	The following theorem is a consequence of the above discussion.
	
	\begin{theorem}\label{c}
		There is an equivalence between the category $REXT$ of ring center extensions and the category $RFAC$ of ring center factor systems. 
	\end{theorem}
   
Let $$ \mathcal{E}(H, K) \equiv	{0}\longrightarrow  H\overset{i} \longrightarrow P\overset{\beta} \longrightarrow  K \longrightarrow {0}$$ be a ring center extension of $H$ by $K$. Let $t$ and $s$ be two sections of $\mathcal{E}(H, K)$. Then there exists a vector space homomorphism $g: K \to H$ such that $$ s(x)= t(x) + g(x).$$
For $x\in K$ and $h\in H,$ we have 
$$\Gamma^s_x(h)= [s(x),h]= [t(x) + g(x), h]= [t(x), h] + [g(x), h].$$ Hence, $$\Gamma^s_x(h)= \Gamma^t_x(h) + ad(g(x))(h).$$ Thus, the vector space homomorphism $\Gamma^t: K \to Der(H)$ depend on the choice of section $t$.\\ 
Now, for each $h\in H$, define the vector space homomorphism $ad_h: H \to H$  by $ad_h(k)=[h,k]$. It is easy to see that $ad_h$ is a Poisson algebra derivation (called as inner derivation) and the set $$IDer(H)= \{ad_h \in Der(H) \ | \ h\in H \}$$ is an ideal of $Der(H).$ Now, the induced map from $K \to OutDer(H)=Der(H)/IDer(H)$ defined by $x\mapsto \Gamma_x^t + IDer(H)$ is independent of the choice of section $t$,
 %$$x\mapsto \Gamma_x^t + IDer(H),\ \ \ K \to OutDer(H)=Der(H)/IDer(H),$$
 where $OutDer(H)$ is the commutative Poisson algebra of outer derivations.  
Consequently, for each ring center extension $\mathcal{E}(H, K)$, we have a Poisson algebra homomorphism
$\Gamma_{\mathcal{E}}: K \to OutDer(H)$ defined as $$ \ \Gamma_{\mathcal{E}}(x)= \Gamma_x^t + IDer(H).$$

\begin{lemma}
Let $\mathcal{E}(H, K)$ be a ring center extension and $t$ be a section of $\mathcal{E}(H, K)$. Then the map $\Gamma_\mathcal{E}: K \to OutDer(H)$ is independent of the choice of sections, as well as the choice of the representative $\mathcal{E}(H, K)$ of the equivalence class.
\end{lemma}
\begin{proof}
    From the above discussions, it is clear that $\Gamma_\mathcal{E}$ is independent of the choice of sections.
Now, let $\mathcal{E}'(H, K)$ be an equivalent ring center extension to $\mathcal{E}(H, K)$. Let $t'$ be a section of $\mathcal{E}'(H,K)$, and $\Gamma_{\mathcal{E}'}: K \to OutDer(H)$ be the corresponding Poisson algebra homomorphism and $(I_H,\phi, I_K)$ be an equivalence from $\mathcal{E}(H,K)$ to $\mathcal{E'}(H,K)$, that is,  we have the following commutative diagram:
$$ 	\xymatrix{\mathcal{E}(H, K) \equiv 0 \ar[r] & H \ar[d]_-{I_H} \ar[r]^-{i} & P \ar[d]_-{\phi} \ar[r]^-{\beta} & K \ar[d]^-{I_K} \ar[r] & 0 \\
		\mathcal{E'}(H, K) \equiv  0 \ar[r] & H \ar[r]_-{i'} & P' \ar[r]_-{\beta'} & K \ar[r] & 0. }$$ 
Since $t$ is a section of $\mathcal{E}(H,K)$, $\phi \circ t$ is a section of $\mathcal{E'}(H,K)$. Therefore, there exists a vector space homomorphism $g$ from $K$ to $H$ such that $t'(x)= g(x) + \phi(t(x))$. Now, we have 

\begin{align*}
    \Gamma^{t'}_x(h)+ IDer(H)&=[t'(x),h]+ IDer(H)= [g(x) + \phi(t(x)), h]+IDer(H)\\&= [\phi(t(x)), h]+IDer(H) \\& = [\phi(t(x)), \phi(h)]+IDer(H)\\&= \phi([t(x),h])+IDer(H)\\&=\phi(\Gamma^t_x(h))+IDer(H)\\&= \Gamma^t_x(h) +IDer(H)
\end{align*}
(since $\phi(h)=h$ for all $h \in H$).
Thus, $\Gamma_{\mathcal{E'}}(x)=\Gamma_{\mathcal{E}}(x)$ for all $x \in K$. Therefore, $\Gamma_{\mathcal{E}}$ is independent of the choice of sections and the equivalent extensions.
\end{proof}

\begin{definition}
   A Poisson algebra homomorphism from $K$ to $OutDer(H)$ is called a \textbf{coupling} or an \textbf{abstract kernel} of $K$ to $H$.
\end{definition}
Based on the above discussions, we have seen that for any ring center extension $\mathcal{E}(H, K)$ we have an abstract kernel $\Gamma_{\mathcal{E}}$ of $K$ to $H$. Now, we have the following theorem:
\begin{theorem}
    Let $RExt(H, K)$ denote the set of all equivalence classes of ring center extensions in $\mathcal{E}(H, K)$ and $Hom(K, OutDer(H))$ denote the set of all abstract kernels (couplings) of $K$ to $H$. Then there is a natural map $\Phi$ from $RExt(H, K)$ to $Hom(K, OutDer(H))$ defined by $\Phi([\mathcal{E}])=\Gamma_{\mathcal{E}}$.
\end{theorem}
\begin{remark}
 The map $\Phi$ described in the above theorem is called the abstract kernel map.
 The abstract kernel map is not necessarily injective, that is, two non-equivalent ring center extensions of $H$ by $K$ may induce the same abstract kernels of $K$ to $H$.  
\end{remark}
\begin{example}
 Consider the Lie algebras $H=\mathbb{R}e_1$, $K= (span\{e_2,e_3\}, [, ~])$,  $P_1 = (\R^3, [,~ ]_1)$ and $P_2 = (\R^3, [,~ ]_2)$, where $[, ~]$, $[,~]_1$ and $[,~ ]_2$ are defined as follows: $$[e_2,~ e_3]=e_2,$$ $$[e_2,~ e_3]_1=e_2, [e_1,~ e_3]_1=0=[e_1,~ e_2]_1,$$ $$[e_2,~ e_3]_2=e_2, [e_1,~ e_3]_2=e_1, [e_1,~ e_2]_2 = 0,$$
where  $e_1=(1,0,0),e_2=(0,1,0), e_3=(0,0,1)\in \R^3$.
Clearly, $H$ is an ideal of both $P_1$ and $P_2$. Also, $P_1/H \cong K$ and $P_2/H \cong K.$
Since $[P_1,~ P_1]_1 = \mathbb{R}e_1$ and $[P_2,~ P_2]_2 = span\{e_1,e_2\}$, $P_1$ and $P_2$ are not isomorphic. So, ring center extensions  $\mathcal{E}(H, K) \equiv	{0}\longrightarrow  H\overset{i} \longrightarrow P_1\overset{\pi} \longrightarrow  K \longrightarrow {0}$ and $ \mathcal{E}'(H, K) \equiv	{0}\longrightarrow  H\overset{i'} \longrightarrow P_2\overset{\pi'} \longrightarrow  K \longrightarrow {0}$ are not equivalent. 

Since every Lie algebra is a commutative Poisson algebra with trivial ring structure,  $H$ and $K$ are commutative Poisson algebras.  So, we have $Der(H)=0$. Thus,  $OutDer(H)=0$. Therefore, the abstract kernel $\Gamma_{\mathcal{E}}: K \to OutDer(H)$ is a zero map for any ring center extension of $H$ by $K$. Hence, $\Phi$ is not injective.
\end{example}
\begin{theorem}\label{4.7}
 Let $H$ be a commutative Poisson algebra with trivial ring structure such that $Z(H) = \{0 \}.$ Then the map $\Phi$
 from $RExt(H, K)$  to the set $Hom(K, OutDer(H))$ is bijective. More explicitly, every abstract kernel $\eta$ of $K$ to $H$  is uniquely determined by an equivalence class of extensions in $RExt(H, K)$.  
\end{theorem}
\begin{proof}
   Let $ \eta \in Hom(K, OutDer(H))$ be an abstract kernel of $K$ to $H$. Consider the Pull Back Diagram 
   \[
\xymatrix@C=1.5cm@R=1.5cm{
   L \ar[d]_{p_1} \ar[r]^{p_2} & K \ar[d]^{\eta} \\
   Der(H) \ar[r]_{\pi} & OutDer(H)
}
\]
 where, $$L = \{( \Gamma, x) \ |  \
\Gamma \in Der(H) \ \text{and} \ \Gamma + IDer(H) = \eta(x) \}.$$ Clearly, $L$ is the Poisson subalgebra of the direct product $Der(H) \times K$. Here, $p_1$ and $p_2$ denote the canonical projections, and it is easy to see that $p_2: L \to K$  is a surjective Poisson algebra homomorphism. Its kernel is
 $$kerp_2= \{ (\Gamma, 0) \ | \ \Gamma + IDer(H)= \eta(0)= IDer(H)\}= IDer(H) \times \{0\}.$$
Since  $Z(H) = \{0 \},$ the map $\alpha$ from $H$ to $L$ defined by   $ \alpha(h)=(ad_h, 0) $  
is an injective Poisson algebra homomorphism with $image \ \alpha=ker \ p_2.$  This gives a ring center extension $\mathcal{E}$ of $H$ by $K$ given by the exact sequence 
$$ \mathcal{E}(H, K) \equiv	{0}\longrightarrow  H\overset{\alpha} \longrightarrow L\overset{p_2} \longrightarrow  K \longrightarrow {0.}$$
By the axiom of choice, there is a map $\theta$ from $K$ to $Der(H)$ such that $\theta(x) + IDer(H)=\eta(x)$. This determines a section $$t: K \to L, \ \ \ t(x)=(\theta(x), x). $$
The corresponding abstract kernel  $\Phi([\mathcal{E}])$ associated with the ring center extension $\mathcal{E}$ is given by $\Phi([\mathcal{E}])(x)= \Gamma_x^t + IDer(H),$
where $\Gamma_x^t$ is defined by $[t(x),\alpha(h)]= \alpha(\Gamma^t_x(h)).$
Explicitly,
$$ \alpha(\Gamma^t_x(h))= [(\theta(x),x),(ad_h,0)]=([\theta(x),ad_h],0)= (ad_{\theta(x)(h)},0)= \alpha(\theta(x)(h)).$$
Thus, $\Gamma^t_x(h)=\theta(x)(h)$ for all $h\in H,$ so $\Gamma^t_x=\theta(x)$. Consequently, $$\Phi([\mathcal{E}])(x)= \Gamma^t_x+ IDer(H)=\theta(x)+ IDer(H)=\eta(x)$$ for all $x \in K$. This proves that $\Phi$ is surjective.

To prove injectivity, suppose that $\Phi([\mathcal{E}_1])= \Phi([\mathcal{E}_2])$, where  
$$ \mathcal{E}_1 \equiv	{0}\longrightarrow  H\overset{\alpha_1} \longrightarrow L_1\overset{\beta_1} \longrightarrow  K \longrightarrow {0}, \ \ \ \  \mathcal{E}_2 \equiv	{0}\longrightarrow  H\overset{\alpha_2} \longrightarrow L_2\overset{\beta_2} \longrightarrow  K \longrightarrow {0}$$ are two ring center extensions of $H$ by $K$. Let $t_1,$ $t_2$ be sections of $\mathcal{E}_1$ and $ \mathcal{E}_2,$ respectively, with corresponding ring center factor systems $( H, K,  f^{t_1}, h^{t_1}, \Gamma^{t_1})$ and  $(H, K,  f^{t_2}, h^{t_2}, \Gamma^{t_2})$. 
Now, we have
$$\Gamma^{t_1}_x + IDer(H)=\Gamma^{t_2}_x + IDer(H), $$
for all $x\in K.$ Since $Z(H)= \{0\},$ there exists a unique vector space homomorphism $g: K \to H$ such that 
\begin{align} \label{3.18}
    \Gamma^{t_1}_x=ad_{g(x)} + \Gamma^{t_2}_x
\end{align}
for all $x\in K.$ Now by equation $(\ref{3.3})$, we have
\begin{equation} \label{3.19}
\begin{rcases}
    \Gamma^{t_1}_{xy}(h) &= [h, f^{t_1}(x, y)]\\
   \Gamma ^{t_1}_{[x, y]} (h) &= \Gamma^{t_1}_x \Gamma^{t_1}_y(h) - \Gamma^{t_1}_y \Gamma^{t_1}_x(h) - [h^{t_1}(x, y), h] 
\end{rcases}
 \end{equation}
   and
   \begin{equation} \label{3.20}
   \begin{rcases}
       \Gamma^{t_2}_{xy}(h) &= [h, f^{t_2}(x, y)]\\  
    \Gamma ^{t_2}_{[x, y]} (h) &= \Gamma^{t_2}_x \Gamma^{t_2}_y(h) - \Gamma^{t_2}_y \Gamma^{t_2}_x(h) - [h^{t_2}(x, y), h]
   \end{rcases}  
 \end{equation}
   for all $x,y \in K$ and $h\in H$. Using equations $(\ref{3.18})$, $(\ref{3.19})$, $(\ref{3.20})$, and the fact that $Z(H)=\{0\}$, it can be shown that equation $(\ref{x})$ holds with $\nu= I_K$ and $\lambda=I_H$. It follows that $(I_H, g, I_K)$ is an equivalence between the ring center factor system $( H, K,  f^{t_1}, h^{t_1}, \Gamma^{t_1})$ and $(H, K, f^{t_2}, h^{t_2}, \Gamma^{t_2})$. By Theorem $(\ref{c})$, it follows that extensions $\mathcal{E}_1$ and $\mathcal{E}_2$ are equivalent. This completes the proof. 
\end{proof}

Now, we discuss the following problem:

\begin{problem}
Let $H$ be an abelian Poisson algebra. Classify all ring center extensions of $H$ by $K$ (up to equivalence) with the given abstract kernel $ \Gamma$.  
\end{problem} 
\begin{remark}
    Let $H$ be an abelian Poisson algebra, and $K$ be a commutative Poisson algebra. Then the abstract kernel $\Gamma$ is a Poisson algebra homomorphism from  $K$ to $Der(H)$.
\end{remark}
\begin{definition} \label{D3.5}
    Let $H$ be an abelian Poisson algebra, and $K$ be a commutative Poisson algebra. Then the triple $(f, h,\Gamma )$ is said to be a $2$-cocycle of a commutative Poisson algebra $K$ with coefficient in an abelian Poisson algebra $H$, if the triple $(f, h, \Gamma)$ is such that $f$ is symmetric bilinear and $h$ is an alternating maps from $K \times K$ to $H$, and $\Gamma: K \to Der(H)$ is Poisson algebra homomorphism satisfying the following conditions for all $x,y,z \in K$ and $h\in H$:
    \begin{enumerate}
    \item $\Gamma_{xy}(h) = 0.$
    \item $f(x y, z) = f(x,y z).$
    \item $ h([x,y], z) + h([y, z], x) + h([z,x], y)=\Gamma_z(h(x, y)) + \Gamma_x(h(y, z)) + \Gamma_y(h(z, x)).$
    \item $h(xy,z)  =  \Gamma_z(f(x,y))  + f([x,z],y) + f(x,[y,z]).$
\end{enumerate}
\end{definition}

\begin{remark}
    Let $\mathcal{E}(H,K)$ be a ring center extension of an abelian Poisson algebra $H$ by a commutative Poisson algebra $K$ with a choice of section $t$. Then from the above discussion, it is clear that we have a $2$-cocycle $(f^t, h^t,\Gamma )$.

    Conversely, if we have a $2$-cocycle $(f, h,\Gamma )$ of a commutative Poisson algebra $K$ with coefficient in an abelian Poisson algebra $H$, then from Proposition $\ref{P3.2}$, it can be seen that there exists a ring center extension $\mathcal{E}(H,K)$ of $H$ by $K$ with a choice of section $t$ such that $f^t=f, h^t=h$ and $\Gamma_{\mathcal{E}}=\Gamma$.
\end{remark}

\begin{definition} \label{F}
 Two $2$-cocycle $(f^s, h^s,\Gamma )$ and $(f^t, h^t,\Gamma )$ of the ring center extension $\mathcal{E}(H,K)$ are said to be equivalent if there exists a vector space homomorphism $g: K \to H$ satisfying
    \begin{enumerate}
        \item $f^s(x,y)= - g(xy) + f^t(x,y)$
        \item $ h^s(x, y)=\Gamma_{x}(g(y)) - \Gamma_{y}(g(x)) - g([x,y]) + h^t(x, y).$
    \end{enumerate}
\end{definition}

Let $Z^2_{ \Gamma}(K,H)$ denote the set of all $2$-cocycles $(f, h, \Gamma)$ of a commutative Poisson algebra $K$ with coefficients in an abelian Poisson algebra $H$ with the abstract kernel $\Gamma$.

We define an operation on $Z^2_{ \Gamma}(K,H)$ by $$(f, h, \Gamma) + (f', h', \Gamma) = (f+f', h+h', \Gamma).$$ Since $H$ is abelian, it follows immediately that $(f+f', h+h', \Gamma) \in Z^2_{ \Gamma}(K,H) $. Moreover, for every $(f, h, \Gamma)\in Z^2_{ \Gamma}(K,H),$ element $(-f, -h, \Gamma)$ also belongs to  $Z^2_{ \Gamma}(K,H).$ Hence, $Z^2_{ \Gamma}(K,H)$ is an abelian group with respect to the operation defined above. \\
Now, let $B^2_{ \Gamma}(K,H)$ denote the set of all $2$-cocycles that are equivalent to the trivial $2$-cocycle $(f_0, h_0, \Gamma).$ More precisely, by Definition \ref{F}, we have $(f, h, \Gamma)\in B^2_{ \Gamma}(K,H) $  if and only if there exists a vector space homomorphism $g: K \to H$ such that $$f(x,y)= - g(xy), $$ $$ h(x, y)=\Gamma_{x}(g(y)) - \Gamma_{y}(g(x)) - g([x,y]).$$ 
Note that, for any vector space homomorphism $g: K \to H$, the pair $(g^*, \partial g,\Gamma )$ is a member of $Z^2_{ \Gamma}(K,H)$, where $g^*$ and $\partial g$ are maps from $K \times K$ to $H$ given by $g^* (x, y) = - g(xy)$ and $\partial g(x, y) = \Gamma_{x}(g(y)) - \Gamma_{y}(g(x)) - g([x,y])$, respectively.
The elements of $B^2_{ \Gamma}(K,H)$ are called the $2$-coboundaries of $K$ with coefficient in $H$ associated with the abstract kernel $ \Gamma$. The quotient group $$ H^2_{ \Gamma}(K,H) = Z^2_{ \Gamma}(K,H)/B^2_{ \Gamma}(K,H)$$ is called the second cohomology group of $K$ with coefficient in $H$ associated with the abstract kernel $\Gamma$.     

\begin{theorem}
Let $H$ be an abelian Poisson algebra, and $K$ be a commutative Poisson algebra. Let $ \Gamma$ be an abstract kernel of $K$ to $H.$ Then there is a natural bijective correspondence between the set $RExt_{ \Gamma}(H, K)$ of equivalence classes of ring center extensions of $H$ by $K$ with the given abstract kernel $ \Gamma$ and the second cohomology group ${H^2}_{ \Gamma}(K, H).$  
    
\end{theorem}
\begin{proof}
Let  $\mathcal{E}(H, K)$ and  $\mathcal{E'}(H, K)$ be equivalent ring center extensions of $H$ by $K$ with the abstract kernel $ \Gamma$. Let $(f^t,h^t,\Gamma)$ and $(f^{t'} ,h^{t'},\Gamma)$ be 2-cocycles corresponding to sections $t$ and $t'$ of $\mathcal{E}(H, K)$ and  $\mathcal{E'}(H, K)$, respectively. Then there exists a vector space homomorphism $g: K \to H$ such that
  $$f^{t'}(x,y)= -g(xy) + f^t(x,y),$$ 
  $$h^{t'}(x,y)= \Gamma_x(g(y)) - \Gamma_y(g(x)) - g([x,y]) + h^t(x,y).$$ 
  This implies that 
  $$(f^t, h^t,\Gamma)+ B^2_{ \Gamma}(K,H)=(f^{t'},h^{t'},\Gamma)+ B^2_{ \Gamma}(K,H).$$ 
  Therefore, we have a map 
  $\eta: RExt_{ \Gamma}(H,K) \to H^2_{ \Gamma}(K,H)$ 
  defined by $\eta([\mathcal{E}])=(f^t, h^t,\Gamma)$ $+ B^2_{ \Gamma}(K,H),$
  where $t$ is a section of $\mathcal{E}(H, K)$.
  Now, let $(f, h,\Gamma)\in Z^2_{ \Gamma}(K,H).$ Then by Proposition $\ref{P3.2}$, there exists a ring center extension $\mathcal{E}(H, K)$ of $H$ by $K$ with section $t$ such that $f^t=f$, $h^t=h$ and $\Gamma_{\mathcal{E}}=\Gamma$. Hence, $\eta$ is surjective. \\
  Let $\mathcal{E}(H, K)$ and $\mathcal{E'}(H, K)$ be ring center extensions of $H$ by $K$ with sections $t$ and $t'$, respectively, and abstract kernel $ \Gamma$ such that 
  $\eta([\mathcal{E}])=\eta([\mathcal{E'}]),$
  that is,
  $(f^t, h^t,\Gamma)+ B^2_{ \Gamma}(K,H)=(f^{t'},h^{t'},\Gamma)+ B^2_{ \Gamma}(K,H).$ 
  Then, there exists a vector space homomorphism $g: K \to H$  such that
  $$f^{t'}(x,y)= -g(xy) + f^t(x,y),$$  
  $$h^{t'}(x,y)= \Gamma_x(g(y)) - \Gamma_y(g(x)) - g([x,y]) + h^t(x,y).$$ 
  Thus, the ring center factor systems $(H,K,f^t,h^t, \Gamma)$ and $(H,K,f^{t'},h^{t'}, \Gamma)$ 
  are equivalent. Therefore, $\mathcal{E}(H, K)$ and $\mathcal{E'}(H, K)$ are equivalent. So $\eta$ is injective.
\end{proof}
Let $H$  be a commutative Poisson algebra with trivial ring structure (not necessarily abelian), and $K$ be a commutative Poisson algebra. Suppose $ \Gamma: K \to OutDer(H)$ is an abstract kernel. If $\delta \in Der(H)$ and
$h \in Z(H),$ then $\delta([h, k]) = [\delta(h), k] + [h, \delta(k)]$ implies $[\delta(h), k] = 0, $ and also $\delta(h)k = 0$ for all $k \in  H.$ Thus
 $\delta(h) \in Z(H).$ Hence, a Poisson algebra homomorphism $\chi : Der(H) \to Der(Z(H)) $ is induced, which is defined by $\chi(\delta) = \delta|_{Z(H)}$.
\begin{proposition}
    Let $\sigma: K \to Der(H)$ be a vector space homomorphism, which is a
lifting of $\Phi$ in the sense that $\nu \circ \sigma = \Gamma,$ where $\nu: Der(H) \to OutDer(H)$ is the quotient map. Then $\chi$
induces a map $ \overline{\chi}$ from $Hom(K, OutDer(H))$ to the set $Hom(K, Der(Z(H)))$ given by $\overline{\chi}(\Gamma) = \chi \circ \sigma $.
\end{proposition}
\begin{proof}
 Since $\Gamma: K \to OutDer(H)$ is a Poisson algebra homomorphism, we have
 $$\sigma ([x, y]) + IDer(H) = [\sigma(x), \sigma(y)] + IDer(H), \ \ \ \sigma(xy) + IDer(H) = IDer(H).$$
 Hence, there exist maps $f, g: K\times K \to H$  such that $$\sigma ([x, y])  = ad_{f(x, y)} + [\sigma(x), \sigma(y)], \ \ \ \sigma(xy)= ad_{g(x, y)}.$$ It follows that $$(\chi \circ \sigma) [x, y] = [(\chi \circ \sigma) (x), (\chi \circ \sigma)(y)], \ \ \ (\chi \circ \sigma)(xy) = (\chi \circ \sigma)(x) (\chi \circ \sigma)(y)$$ for all $x, y\in K.$ Thus, $\chi \circ \sigma$  is a Poisson algebra homomorphism from $K$ to $Der(Z(H))$. Now, let $\sigma'$ be another lifting of $\Gamma.$ Then $\sigma(x) + IDer(H) = \sigma'(x) + IDer(H)$ for all $x\in K.$ Therefore, there exists a vector space homomorphism $q : K \to H$ such that $\sigma(x) = ad_{q(x)} + \sigma'(x).$  It follows that $(\chi \circ \sigma)(x) = (\chi \circ \sigma')(x).$ Hence, $\chi \circ \sigma$ depends only on $\Gamma.$  Since $Z(H)$ is an abelian Poisson algebra, $Der(Z(H)) = OutDer(Z(H)).$ Thus, $\chi$ induces a map $$\overline{\chi} : Hom(K, OutDer(H)) \to Hom(K,  Der(Z(H)))$$  defined by $\overline{\chi}(\Gamma) = \chi \circ \sigma $.      
\end{proof}

\begin{proposition} \label{3.12}
 Let $$ \mathcal{E} \equiv	{0}\longrightarrow  H \overset{\alpha} \longrightarrow P \overset{\beta} \longrightarrow  K \longrightarrow {0}$$ and   $$ \mathcal{E'} \equiv	{0}\longrightarrow  H\overset{\alpha'} \longrightarrow P' \overset{\beta'} \longrightarrow  K \longrightarrow {0}$$   
be two ring center extensions of a commutative Poisson algebra $H$ with trivial ring structure (not necessarily abelian) by a commutative Poisson algebra $K$ such that $\Gamma_{\mathcal{E}} = \Gamma_{\mathcal{E'}} = \Gamma. $ Then there is a section $t$ of $\mathcal{E}$ and a section $t'$ of $\mathcal{E'}$ such that $\Gamma^t = \Gamma^{t'} = \overline{\chi}(\Gamma),$ and $- f^t(x, y) + {f^{t'}}(x, y)$ and $- h^t(x, y) + {{h^{t'}}(x, y)}$ belongs in $Z(H)$ for all $x, y \in K.$ If the maps $f,h: K \times K \to Z(H)$ defined by $$f(x, y) = - f^t(x, y) + {f^{t'}}(x, y), \ \ \ h(x, y) = - h^t(x, y) + {{h^{t'}}(x, y)} $$ for all $x,y \in K$ then $(f,h,\overline{\chi}(\Gamma)) \in Z^2_{\overline{\chi}(\Gamma)}(K, Z(H)).$  
\end{proposition}
\begin{proof}
    Let $t$ be a section of $\mathcal{E}$ and $s$ be a section of $\mathcal{E'}$. Since $\Gamma_{\mathcal{E}} = \Gamma_{\mathcal{E'}},$ we have $$\Gamma^t(x) + IDer(H) = \Gamma^{s}(x) + IDer(H)$$
  for all $x\in K.$ Thus, there exists a vector space homomorphism $g : K \to H $ such that $\Gamma^t(x) = ad_{g(x)} + \Gamma^{s}(x)$ for all $x\in K.$ Define a map $t' : K \to P'$ by $t'(x) = g(x) + s(x).$ Then $t'$  is also a section of $\mathcal{E'},$ and moreover, $\Gamma^{t'}(x) = ad_{g(x)} + \Gamma^{s}(x) = \Gamma^{t}(x) $  for all $x\in K.$ Hence, $\Gamma^{t} = \Gamma^{t'}.$ Now 
  $$f^t(x, y)= t(x) t(y) - t(x y),  \ \ \  f^{t'}(x, y)= t'(x) t'(y) - t'(x y), $$ $$ h^t(x, y) = [t(x), t(y)] - t([x, y]),  \ \ \ h^{t'}(x, y) = [t'(x), t'(y)] - t'([x, y]).$$ Hence, $ad_{f^t(x, y)} = - \Gamma^t_{xy} = - \Gamma^{t'}_{xy} = ad_{{f^{t'}}(x, y)} $ and $ ad_{h^t(x, y)} = [\Gamma^t_{x}, \Gamma^t_{y} ] - \Gamma^t_{[x, y]} = [\Gamma^{t'}_{x}, \Gamma^{t'}_{y} ] - \Gamma^{t'}_{[x, y]} = ad_{{h^{t'}}(x, y)} $ for all $x, y \in K.$  Thus, $ad_{- f^t(x, y) + {{f^{t'}}(x, y)} } = ad_{- h^t(x, y) + {{h^{t'}}(x, y)} } = 0.$ This shows that $- f^t(x, y) + {{f^{t'}}(x, y)}$ and  $ - h^t(x, y) + {{h^{t'}}(x, y)}$ belongs in $Z(H).$ Put  $$f(x, y) = - f^t(x, y) + {f^{t'}}(x, y), \  \text{and}\ \ h(x, y) = - h^t(x, y) + {{h^{t'}}(x, y)}. $$ Then, it is straightforward to verify that $(f, h,\overline{\chi}(\Gamma)) \in Z^2_{\overline{\chi}(\Gamma)}(K, Z(H)).$
\end{proof}

\begin{theorem}\label{4.13}
    Let $\Gamma: K \to OutDer(H)$ be an abstract kernel from $K$ to $H$, where $H$ is a commutative Poisson algebra with trivial ring structure (not necessarily abelian) and $K$ is a commutative Poisson algebra. Then the second cohomology group $H^2_{\overline{\chi}(\Gamma)}(K, Z(H))$ acts sharply transitively on the set $RExt_{\Gamma}(H, K)$ of equivalence classes of $H$ by $K$ associated with the abstract kernel $\Gamma$.
\end{theorem}
\begin{proof}
    Let $\mathcal{E}$ be a ring center extension of $H$ by $K$ associated with the abstract kernel $\Gamma$. Let $t$ be a section of $\mathcal{E}$ with corresponding ring center factor system $(H, K, f^t, h^t, \Gamma^t )$. Then $\Gamma(x)=\Gamma^t_x + IDer(H)$ for all $x \in K$. Let $(f,h,\overline{\chi}(\Gamma)) \in Z^2_{\overline{\chi}(\Gamma)}(K, Z(H)).$ It is easily seen that $(H, K, f^t +f, h^t + h,  \Gamma^t)$ is again a ring center factor system. Let $\mathcal{E} \star (f, h, \overline{\chi}(\Gamma))$ denote the corresponding ring center extension. Clearly, $\mathcal{E} \star (f, h, \overline{\chi}(\Gamma))$ is also associated with the abstract kernel $\Gamma$. Let $(f', h',\overline{\chi}(\Gamma))$ be another 2-cocycle in $Z^2_{\overline{\chi}(\Gamma)}(K, Z(H))$ such that the cohomology class 
    $$[(f, h, \overline{\chi}(\Gamma))]= (f, h, \overline{\chi}(\Gamma)) + B^2_{\overline{\chi}(\Gamma)}(K, Z(H))= [(f', h',\overline{\chi}(\Gamma))] $$ $$=(f', h',\overline{\chi}(\Gamma)) + B^2_{\overline{\chi}(\Gamma)}(K, Z(H)).$$
    Then, there is a vector space homomorphism $g:K \to Z(H) \subseteq H$ such that $$f(x,y)= g^{\ast}(x,y) + f'(x,y)$$ and $$h(x,y)=\partial g(x,y)+ h'(x,y)$$ for all $x,y \in K$. Clearly, $f^t + f = f^t + f' +g^{\ast} $ and $h^t + h= h^t + h' + \partial g$. Since $(g^{\ast}, \partial g, \overline{\chi}(\Gamma)) \in B^2_{\overline{\chi}(\Gamma)}(K, Z(H))$. Hence, the ring center factor system $(H, K, f^t +f, h^t + h, \Gamma^t)$ is equivalent to $(H, K, f^t + f', h^t + h', \Gamma^t)$. This shows that their equivalence class will be equal, that is, $[\mathcal{E}\star (f,h,\overline{\chi}(\Gamma))]= [\mathcal{E}\star (f',h',\overline{\chi}(\Gamma))]$. \\
Let $\mathcal{E}$ and $\mathcal{E}'$ be equivalent ring center extensions of $H$ by $K$ associated with the abstract kernel $\Gamma$ and $(f,h, \overline{\chi}(\Gamma)) \in Z^2_{\overline{\chi}(\Gamma)}(K, Z(H))$. 
    %By Theorem $(\ref{3.3})$ 
Then by Theorem $\ref{c}$, we have a section $t$ of $\mathcal{E}$ and a section $t'$ of $\mathcal{E}'$ such that the ring center factor systems say, $(H, K, f^t, h^t, \Gamma^t)$ and $(H, K, f^{t'}, h^{t'}, \Gamma^{t'})$ are equivalent. Hence, there is a vector space homomorphism $g: K \to H$ such that $$f^t(x,y)= g^{\ast}(x,y) + f^{t'}(x,y)$$ and $$h^t(x,y)= \partial g(x,y) + h^{t'}(x,y)$$ for all $x,y \in K$. Clearly, the ring center factor systems $( H, K, f^t+ f, h^t+ h, \Gamma^t)$ and $( H,K, f^{t'}+ f, h^{t'}+ h, \Gamma^{t'})$ are equivalent, that is, $[\mathcal{E} \star (f,h, \overline{\chi}(\Gamma))]=[\mathcal{E}' \star (f,h, \overline{\chi}(\Gamma))]$. Thus, $H^2_{\overline{\chi}(\Gamma)}(K, Z(H))$ acts on $RExt_{\Gamma}(H, K)$ by the action $\star $ defined as $$[\mathcal{E}] \star [(f,h, \overline{\chi}(\Gamma))]:=[\mathcal{E} \star (f,h, \overline{\chi}(\Gamma))].$$ 
Let $\mathcal{E}$ and $\mathcal{E}'$ be ring center extensions associated with the abstract kernel $\Gamma$. Then by Proposition $\ref{3.12}$, there is a section $t$ of $\mathcal{E}$ and a section $t'$ of $\mathcal{E}'$ such that $\Gamma^t= \Gamma^{t'}=\overline{\chi}(\Gamma)$, and there exist maps $f: K \times K \to Z(H)$ and $h: K \times K \to Z(H)$ defined by $f(x,y)= - f^t(x,y) + f^{t'}(x,y)$ and $h(x,y)= -h^t(x,y) + h^{t'}(x,y)$ such that $(f,h, \overline{\chi}(\Gamma)) \in Z^2_{\overline{\chi}(\Gamma)}(K, Z(H))$. Clearly, $[\mathcal{E}] \star [(f,h, \overline{\chi}(\Gamma))]=[\mathcal{E}']$. It shows that $\star$ is transitive. \\
To prove that the action $\star$ is sharply transitive, we need to prove that if $(f,h, \overline{\chi}(\Gamma)) \in Z^2_{\overline{\chi}(\Gamma)}(K, Z(H))$ such that $[\mathcal{E}] \star [(f,h, \overline{\chi}(\Gamma))]=[\mathcal{E}]$, then $(f,h, \overline{\chi}(\Gamma)) \in B^2_{\overline{\chi}(\Gamma)}(K, Z(H))$. Assume that $[\mathcal{E}] \star [(f,h,\overline{\chi}(\Gamma))]=[\mathcal{E}]$, then there exists a section $t$ of $\mathcal{E}$ such that the ring center factor system $(H, K, f^t, h^t, \Gamma^t)$ is equivalent to  $(H, K, f^{t}+ f, h^{t}+ h, \Gamma^{t})$. Hence, there is a vector space homomorphism $g: K \to H$ such that $$f^t(x,y)+ f(x,y) = g^{\ast}(x,y) + f^{t}(x,y),$$ and $$h^t(x,y) + h(x,y) = \partial g(x,y) + h^{t}(x,y)$$ for all $x,y \in K$, and $\Gamma_x^t(h)= ad_{g(x)}(h)+ \Gamma_x^t(h)$ for all $x \in K$ and $h \in H$. This implies $g(x)\in Z(H)$ for all $x \in K.$ Thus $f(x,y)= -g(xy)$ and $h(x,y)= \Gamma_x^t(g(y))- \Gamma_y^t(g(x)) -g([x,y])$ for all $x,y \in K$. This shows that $f=g^{\ast}$ and $h=\partial g$, where $g$ is a vector space homomorphism from $K$ to $Z(H)$. It follows that $(f,h,\overline{\chi}(\Gamma)) \in B^2_{\overline{\chi}(\Gamma)}(K, Z(H))$. This completes the proof.
\end{proof}
\begin{corollary}
Let $H$ be a commutative Poisson algebra with trivial ring structure (not necessarily abelian) and $K$ be a commutative Poisson algebra. If there is a ring center extension of $H$ by $K$ associated with the abstract kernel $\Gamma$. Then there exists a bijective correspondence between $RExt_{\Gamma}(H, K)$ and $H^2_{\overline{\chi}(\Gamma)}(K, Z(H))$.
\end{corollary}
\begin{proof}
    Define a map $\rho: H^2_{\overline{\chi}(\Gamma)}(K, Z(H)) \to RExt_{\Gamma}(H, K)$ by $$\rho([(f,h,\overline{\chi}(\Gamma))])= [\mathcal{E}] \star [(f,h,\overline{\chi}(\Gamma))].$$ We claim that $\rho$ is a bijection. For injectivity, suppose $\rho([(f,h,\overline{\chi}(\Gamma))])=\rho ([(f',h',\overline{\chi}(\Gamma))])$. Then $$ [\mathcal{E}]\star [(f,h,\overline{\chi}(\Gamma))]=[\mathcal{E}] \star [(f',h', \overline{\chi}(\Gamma))],$$ which implies $$[\mathcal{E}] \star \big([(f,h,\overline{\chi}(\Gamma))]- [(f',h', \overline{\chi}(\Gamma))]\big)=[\mathcal{E}].$$ By Theorem \ref{4.13}, the group $H^2_{\overline{\chi}(\Gamma)}(K, Z(H))$ acts sharply transitively on $RExt_{\Gamma}(H, K)$. Hence, $(f,h,\overline{\chi}(\Gamma))-(f',h',\overline{\chi}(\Gamma))\in B^2_{\overline{\chi}(\Gamma)}(K, Z(H))$, and therefore $[(f,h,\overline{\chi}(\Gamma))]=[(f',h',\overline{\chi}(\Gamma))]$. \\
 For surjectivity, let $[\mathcal{E}] \in RExt_{\Gamma}(H, K)$. By sharp transitivity, there exists a unique class $[(f,h,\overline{\chi}(\Gamma))]\in H^2_{\overline{\chi}(\Gamma)}(K, Z(H))$ such that $$[\mathcal{E}'] \star [(f,h,\overline{\chi}(\Gamma))]  =[\mathcal{E}].$$ Therefore, $\rho([(f,h,\overline{\chi}(\Gamma))])= [\mathcal{E}]$. This establishes surjectivity and completes the proof.

\end{proof}

\noindent{\bf Acknowledgment:} The first-named and third-named authors are thankful to the National Board for Higher Mathematics (NBHM) for providing the project ``Linear Representation of Multiplicative Lie Algebra" (02011/19/2023/NBHM (R.P)/R~\& D-II/5954). 

\noindent{\bf  Data Availability declaration:} Not applicable.

\noindent{\bf Competing Interests:} The authors have no competing interests to declare that are relevant to the content of this article.

\noindent{\bf Funding Declaration:} The first-named author is supported by the National Board for Higher Mathematics (NBHM), project ``Linear Representation of Multiplicative Lie Algebra" (02011/19/2023/NBHM (R.P)/R~\& D-II/5954). The second and third named authors have no relevant financial or non-financial interests to disclose.


\begin{thebibliography}{9}
    \bibitem{AM0} A. L. Agore and G. Militaru, Extending Structures: Fundamentals and Applications, CRC Press, (2019).
    \bibitem{AM}
A. L. Agore and G. Militaru, Jacobi and Poisson algebras, J. Noncommut. Geom. 9 (4) (2015),  1295-1342.
\bibitem{AM1}
A. L. Agore and G. Militaru, The global extension problem, crossed products and co-flag non-commutative Poisson algebras,  J. Algebra 426 (2015), 1-31.

\bibitem{B1}
K. S. Brown,
Cohomology of Groups,
Graduate Texts in Mathematics, Vol. 87, Springer-Verlag (1982).

\bibitem{BY}
Y. H. Bao and Y. Ye, Cohomology structures of a Poisson algebra: I, J. Algebra Appl. 15 (2) (2016), 1650034, 17 pp.

\bibitem{BY1}
Y. H. Bao and Y. Ye, Cohomology structures of a Poisson algebra: II, Sci. China Math. 64 (5) (2021), 903-920.

  \bibitem{CM}
A. J. Calderón Martin, On the structure of split noncommutative Poisson algebras, Linear Multilinear Algebra  60 (2012), 775-785.
\bibitem{CM1}
A. J. Calderón Martin, On extended graded Poisson algebras, Linear Algebra Appl. 439 (2013), 879-892.

\bibitem{PC}
P. Caressa, Examples of Poisson modules I, Rend. Circ. Mat. Palermo (2) 52 (2003), 419-452.
\bibitem{JM}
J. M. Casas, T. Datuashvili, Noncommutative Leibniz-Poisson algebras, Comm. Algebra 34 (2006), 2507-2530. 
\bibitem{GJ}
	G. J. Ellis, On five well-known commutator identities, J. Aust. Math. Soc. (Series A) 54 (1993), 1-19.

\bibitem{EM}
S. Eilenberg and S. MacLane, Cohomology Theory in Abstract Groups I, Ann. of Math. 48 (1) (1947), 51-78.

\bibitem{FL}
D.~R. Farkas and G.~Letzter, Ring theory from symplectic geometry, J. Pure Appl. Algebra 125 (1998), 155-190.

\bibitem{GR}
M. Goze, E. Remm, Poisson algebras in terms of non-associative algebras, J. Algebra 320 (2008), 294-317.

\bibitem{MG} M. Gerstenhaber, The Cohomology Structure of an Associative Ring, Ann. of Math. 78(2) (1963), 267-288.

\bibitem{HS}
G. Hochschild and J.-P. Serre, Cohomology of Lie Algebras, Ann. of Math. 57 (3) (1953), 591-603.

\bibitem{DA}
D. A. Jordan, Finite-dimensional simple Poisson modules, Algebr. Represent. Theory 13 (2010), 79-101.
\bibitem{FK}
F. Kubo, Finite-dimensional non-commutative Poisson algebras, J. Pure Appl. Algebra 113 (1996), 307-314.
\bibitem{FK1}
F. Kubo and F.  Mimura, Extensions of Poisson algebras by derivations, Hiroshima Math. J. 20 (1) (1990), 37-46.

  \bibitem{RL1} R. Lal, Algebra 2, Springer (2017).
	\bibitem{RL} R. Lal, Algebra 4, Springer (2021).

\bibitem{LP} J. L. Loday, T. Pirashvili, Universal enveloping algebras of Leibniz algebras and (co)homology, Math. Ann. 296 (1993), 139-158.

\bibitem{LW}
J. Lü, X. Wang, G. Zhuang, Universal enveloping algebras of Poisson Hopf algebras, J. Algebra 426 (2015), 92-136.
\bibitem{ML}
L. Makar-Limanov, I. Shestakov, Polynomial and Poisson dependence in free Poisson algebras and free Poisson fields, J. Algebra 349 (2012), 372-379. 
  
\bibitem{NR}
A. Nijenhuis and R. W. Richardson, Cohomology and Deformations in Graded Lie Algebras, Bull. Amer. Math. Soc. 72 (1966), 1-29.
  
	\bibitem{MS}	M. S. Pandey and S. K. Upadhyay, Theory of extension of multiplicative Lie algebras, J. Lie Theory 31 (2021), 637-658.


\bibitem{VR} V. Rubtsov, and R. Suchánek, Lectures on Poisson algebras, Tutor. Sch. Workshops Math. Sci., Birkhäuser/Springer, Cham, (2023), 41-116.
\bibitem{UU}
U. Umirbaev, Universal enveloping algebras and universal derivations of Poisson algebras, J. Algebra 354 (2012), 77-94.

\end{thebibliography}
\end{document}